\documentclass[oneside,svgnames,dvipsnames,11pt]{book}

\usepackage[paper=a4paper]{geometry}
\usepackage[utf8]{inputenc}
\usepackage[english]{babel}
\usepackage{graphicx}
\graphicspath{{img/}}
\usepackage{svg}
\usepackage[hidelinks]{hyperref}
\usepackage{amsmath}
\usepackage{amsfonts}
\usepackage{amsthm}
\usepackage{amssymb}
\usepackage{mathtools}
\usepackage{array}
\usepackage{dsfont}
\usepackage{comment}
\usepackage{cleveref}
\usepackage{subcaption}

\usepackage[sectionbib]{natbib}
\usepackage{chapterbib}

\usepackage{minitoc}
\usepackage{url}

\hypersetup{
    colorlinks=true,
    breaklinks=true,
    linkcolor=blue,
    filecolor=magenta,
    urlcolor=cyan,
    citecolor=blue
}

\newtheorem{theorem}{Theorem}[chapter]
\newtheorem{proposition}[theorem]{Proposition}
\newtheorem{corollary}[theorem]{Corollary}
\newtheorem{lemma}[theorem]{Lemma}
\newtheorem{definition}[theorem]{Definition}

\newtheorem{remark}[theorem]{Remark}
\newtheorem{assumption}[theorem]{Assumption}

\newtheorem{conjecture}[theorem]{Conjecture}

\newcommand{\R}{\mathbb{R}}

\usepackage{tikz}
\usetikzlibrary{decorations.pathreplacing}
\usetikzlibrary{fadings}

\usepackage{booktabs}
\usepackage{enumitem}
\usepackage{bbm}

\newcommand{\E}{\mathbb{E}}

\newcommand{\argmin}{\operatornamewithlimits{argmin}}

\usepackage[colorinlistoftodos,textwidth=3cm]{todonotes}

\usepackage{fourier}
\usepackage{marvosym}

\usepackage{tcolorbox}
\tcbuselibrary{theorems}

\definecolor{warning_yellow}{HTML}{fff8c4}
\definecolor{warning_yellow_dark}{HTML}{f2c779}
\newtcolorbox{warningbox}[1]{colback=warning_yellow,
colframe=warning_yellow_dark,fonttitle=\bfseries, sharp corners,
title=#1}
\newtcolorbox{resultbox}[1]{colback=gray!80!white,
colframe=gray!20!white,fonttitle=\bfseries, sharp corners,
title=#1}

\newtcolorbox{mybox_ex}[1]{colback=Emerald!5!white,
colframe=Emerald!75!black,fonttitle=\bfseries,
title=#1}

\newtcbtheorem[auto counter, number within=chapter]{mytheo}{Theorem}%
{colback=RoyalBlue!5!white,colframe=RoyalBlue!90!black,fonttitle=\bfseries, sharp corners, label separator={}}{}
\newtcbtheorem[use counter from=mytheo]{myprop}{Proposition}%
{colback=RoyalBlue!5!white,colframe=RoyalBlue!90!black, sharp corners, fonttitle=\bfseries, label separator={}}{prop}
\newtcbtheorem[use counter from=mytheo]{mydef}{Definition}%
{colback=Emerald!5!white,colframe=Emerald!75!black,fonttitle=\bfseries, sharp corners, label separator={}}{def}
\newtcbtheorem[use counter from=mytheo]{mycorol}{Corollary}%
{colback=RoyalBlue!5!white,colframe=RoyalBlue!90!black,fonttitle=\bfseries, sharp corners, label separator={}}{cor}
\newtcbtheorem[use counter from=mytheo]{mylem}{Lemma}%
{colback=Gray!5!white,colframe=Gray!75!black,fonttitle=\bfseries, sharp corners, label separator={}}{lem}

\newcommand{\cC}{\mathcal{C}}
\newcommand{\cF}{\mathcal{F}}

\newcommand{\cN}{\mathcal{N}}

\usepackage{bm}
\definecolor{blendedblue}{rgb}{0.2,0.2,0.7}

\newcommand{\dd}{\mathop{}\!\mathrm{d}}

\usepackage[f]{esvect}

\usepackage{tikz}
\usetikzlibrary{calc,arrows}

\usepackage{accents}

\newcommand\restr[2]{{
  \left.\kern-\nulldelimiterspace 
  #1
  \littletaller
  \right|_{#2}
  }}

\title{Statistical Analysis of Generative Modeling}
\author{Eddie Aamari \and Arthur St\'ephanovitch}

\begin{document}

\begin{titlepage}
    \centering
    \scshape
    \vspace*{\baselineskip}

    \newcommand{\AuthorBlock}[3]{%
        \begin{minipage}[t]{0.42\textwidth}
            \centering
            {\Large #1\par}
            \vspace{0.35\baselineskip}
            {\normalfont\itshape\small #2\par}
            \vspace{0.25\baselineskip}
            {\normalfont\ttfamily\small
            \href{mailto:#3}{#3}\par}
        \end{minipage}%
    }

    \vfill
    \rule{\textwidth}{1.6pt}\vspace*{-\baselineskip}\vspace*{2pt}
    \rule{\textwidth}{0.4pt}

    \vspace{0.75\baselineskip}

    \mbox{
    \LARGE \hspace{-1.45em}~Statistical~Analysis~of~Markovian~Generative~Modeling}

    \vspace{0.75\baselineskip}

    \rule{\textwidth}{0.4pt}\vspace*{-\baselineskip}\vspace{3.2pt}
    \rule{\textwidth}{1.6pt}

    \vspace{2\baselineskip}

    \large
    Mini-course lecture notes

    \vspace*{3\baselineskip}

    \AuthorBlock
        {Eddie Aamari}
        {Département de Mathématiques et Applications \\ 
        CNRS, École Normale Supérieure, PSL}
        {eddie.aamari@ens.fr}
    \hfill
    \AuthorBlock
        {Arthur St\'ephanovitch}
        {Centre de recherche en économie et statistique \\
        ENSAE, IP Paris}
        {arthur.stephanovitch@ensae.fr}

    \vfill
    \large
    Spring 2026

\end{titlepage}

\frontmatter
\dominitoc
\tableofcontents

\mainmatter

\chapter{Score-based generative models}
\label{chap:score-based-generative-models}
\minitoc

We are deeply grateful to \href{https://www.imo.universite-paris-saclay.fr/~claire.boyer/}{Claire Boyer} for her essential contribution to the first version of this chapter, presented in 2023 at Sorbonne Université.

\section{Stochastic calculus survival kit}

There are two main ways to formalize diffusion-based generative models for quantitative data.
One uses discrete time increments~\cite{ho2020denoising} and requires knowledge of Markov chains only, but it does not yield a clear mathematical framework.
We follow the second approach, which uses continuous time dynamics~\cite{song2020score}. It requires tools from stochastic calculus, but yields a fairly unified functional framework that can then be generalized.

This section gives a minimal overview of stochastic calculus. To make the presentation lighter, we purposely leave all the convergence and measurability issues under the carpet. For rigorous derivations, you shall find all the necessary mathematical details in Jean-François Le-Gall's book~\cite{le2016brownian}.

\subsection{Brownian motion}

Given a measurable space $(E,\mathcal{E})$ and an arbitrary index set $\mathcal{T}$, a \textit{random process indexed by~$\mathcal{T}$ with values in $E$} is a collection $(X_t)_{t \in \mathcal{T}}$ of random variables with values in $E$.

\begin{definition}[Gaussian process]
A (real-valued) random process is called a (centered) \textit{Gaussian process} if any finite linear combination of the variables $(X_t)_{t \in \mathcal{T}}$ is a centered Gaussian.
The distribution of a centered Gaussian process is fully determined by its \textit{covariance kernel}
$K(s,t) := \mathbb{E}[ X_s X_t ]$, for all $s,t \in \mathcal{T}$.
\end{definition}

The main building block of stochastic calculus is the so-called Brownian motion, which we first present in dimension $d=1$.

\begin{definition}[Brownian motion]
\label{def:brownian}
There exists a process $(B_t)_{t \geq 0}$ called \textit{Brownian motion}, which is a centered Gaussian process over $\mathcal{T} = \mathbb{R}_+$ with continuous sample paths $t \mapsto B_t$ and such that any of the following equivalent properties hold.
\begin{itemize}[leftmargin=*]
\item $B_0 = 0$ a.s., and for all $0 \leq s < t$, the random variable $B_t-B_s$ is independent of the $\sigma$-field $\cF_s := \sigma(B_r, r \leq s)$ and distributed according to $\cN(0,t-s)$.
\item $B_0 = 0$ a.s., and for all $0 \leq t_0 < t_1 < \ldots < t_p$, the increments $(B_{t_{j}}-B_{t_{j-1}})_{j}$ are independent and distributed according to $\cN(0,t_j-t_{j-1})$.
\item For all $s,t\geq 0$, $K(s,t) = s \wedge t$.
\end{itemize}
\end{definition}

\begin{proof}
See~\cite[Proposition 2.3]{le2016brownian} for the equivalences. A geometric construction on $\mathcal{T}=[0,1]$ uses Donsker's invariance principle. It is based on an i.i.d.\ sequence $(\xi_i)_{i \in \mathbb{N}}$ of centered real random variables with unit variance. Define the continuous random walk
$W_n(u) := \sum_{i=1}^{\lceil u \rceil} \xi_i \psi(u-i)$
for $u \in [0,1]$, where $\psi(v) := \min\{1,\max\{0,v\}\}$.
Then $(B_t)_{t \in [0,1]}$ is constructed as the limit in distribution of the sequence of scaled processes $\bigl(\frac{1}{\sqrt{n}} W_n(nt)\bigr)_{t \in [0,1]}$.
\end{proof}

\begin{figure}
\centering
\includegraphics[width=1\linewidth]{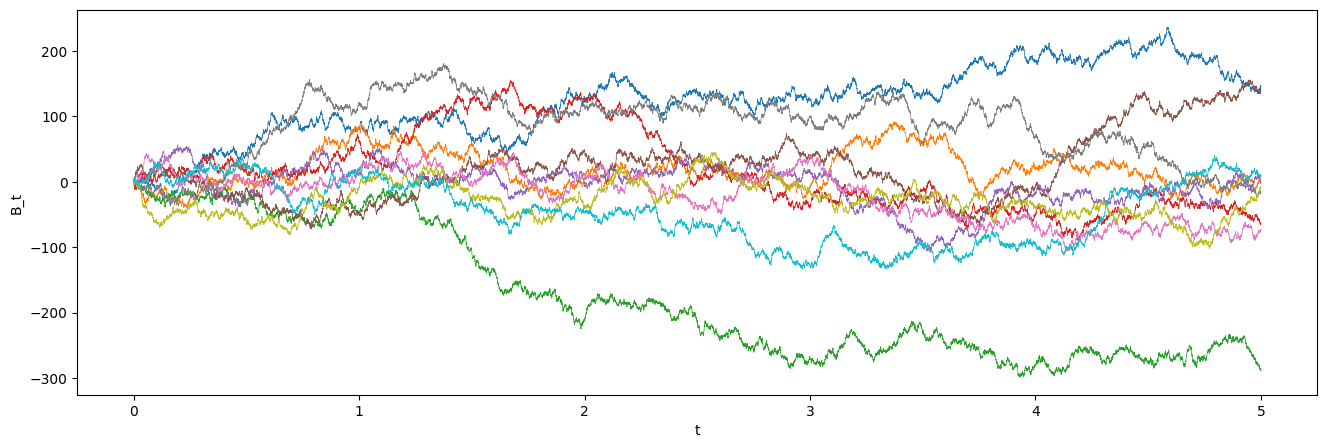}
\caption{Ten trajectories of a Brownian motion.}
\label{fig:brownian}
\end{figure}

See Figure~\ref{fig:brownian} for an illustration of the sample paths of $(B_t)_t$.

\begin{definition}[Filtration, adapted process]
\label{def:filtration-adapted-process}
~
\begin{itemize}[leftmargin=*]
\item A \emph{filtration} over $\mathcal{T} \subset \R$ is an increasing family $(\cF_t)_{t \in \mathcal{T}}$ of $\sigma$-fields, i.e.\ $\cF_s \subset \cF_t$ for all $s \leq t$.
\item A stochastic process $(X_t)_{t \in \mathcal{T}}$ is \emph{adapted} to a filtration $(\cF_t)_{t \in \mathcal{T}}$ if for all $t$, $X_t$ is $\cF_t$-measurable.
\end{itemize}
\end{definition}

Among the many nice properties that the Brownian motion exhibits, let us point out three of the most important ones.

\begin{itemize}[leftmargin=*]
\item \textit{(Martingale property)}
The first characterization of Definition~\ref{def:brownian} yields that the Brownian motion is a martingale adapted to the filtration $\bigl(\cF_s := \sigma(B_r,r\leq s)\bigr)_{s\geq 0}$, since for all $0 \leq s \leq t$:
\begin{align*}
\mathbb{E}[ B_t \mid \cF_s ] = \mathbb{E}[ B_s \mid \cF_s ] + \mathbb{E}[ B_t-B_s \mid \cF_s ] = B_s + \mathbb{E}[ B_t-B_s ] = B_s.
\end{align*}

\item \textit{(Hölder smoothness)}
By definition, a Brownian motion has sample paths that are almost surely continuous. From Kolmogorov’s continuity criterion, they are locally Hölder continuous with exponent $1/2-\delta$ for all $0 < \delta < 1/2$. This essentially comes from the fact that for all $t \geq s \geq 0$, $\mathbb{E} [ (B_t-B_s)^2 ] / (t-s) = 1$.

\item \textit{(Quadratic variation)}
Because sample paths are not more than $1/2$-Hölder everywhere, they do not have finite length. We say that $(B_t)_t$ has infinite \emph{first variation}.
However, its \emph{quadratic variation} is always well defined and deterministic. More precisely, for any sequence of subdivisions $0 = t_0^n < t_1^n < \ldots < t_{p_n}^n = t$ whose maximal spacing tends to zero:
\begin{align*}
\sum_{j=1}^{p_n} (B_{t_j^n} - B_{t_{j-1}^n})^2 \xrightarrow[n \to \infty]{L^2} t.
\end{align*}
\end{itemize}

\subsection{Itô stochastic integral and differential}

Since $(B_t)_t$ exhibits infinite first variation, it is not possible to define the integral $\int_s^t \phi(u) \dd B_u$ as a special case of the usual Stieltjes integral. However, we can construct the Itô integral using the finiteness of its quadratic variation. The construction defines the integral first for elementary (piecewise constant) processes as the sum of weighted increments $\int_a^b X_t \dd B_t := \sum X_{t_{j-1}} (B_{t_{j}}-B_{t_{j-1}})$, and then extends it by density to the space of square-integrable adapted processes. This yields the crucial Itô isometry: $\E[(\int_0^t X_s \dd B_s)^2] = \int_0^t \E[X_s^2] \dd s$.

\begin{definition}[Itô process, stochastic differential]
\label{def:ito-process}
An \emph{Itô process} is a stochastic process $(X_t)_t$ adapted to $(\cF_t)_t$ which can be written as
\begin{align*}
X_t = X_0 + \int_0^t a_s \dd s + \int_0^t b_s \dd B_s,
\end{align*}
where $a_t,b_t$ are continuous stochastic processes in $L^1$ and $L^2$ respectively. We write its \emph{stochastic differential} as
$\dd X_t := a_t \dd t + b_t \dd B_t$.
Here $a_t$ is the \emph{drift} and $b_t$ the \emph{diffusion term} (or \emph{volatility}).
\end{definition}

If $F_t$ is a $\cC^1$ process (meaning $b_t=0$), we recover the classical chain rule $\dd F_t = F'_t \dd t$. For true Itô processes, the non-zero quadratic variation permanently modifies the standard chain rule of calculus.

\begin{theorem}[Multidimensional Itô formula]
\label{thm:ito-multidimension}
Let $(X_t)_{0 \leq t \leq T}$ be an Itô process in $\R^d$
and $\Phi \in \cC^{2,1}(\R^d \times \R_+,\R^k)$.
Then $\bigl(\Phi(X_t,t) \bigr)_{0 \leq t \leq T}$ is an Itô process in $\R^k$ with stochastic differential
\begin{align*}
\dd \Phi(X_t,t)
&= \partial_t \Phi(X_t,t) \dd t
+ \sum_{k=1}^d \partial_{x_k} \Phi (X_t,t) \dd X_t^{(k)}
+ \frac{1}{2} \sum_{k,\ell=1}^d \partial^2_{x_k,x_\ell} \Phi(X_t,t) \dd \langle X^{(k)},X^{(\ell)} \rangle_t,
\end{align*}
where $X_t = (X_t^{(1)},\ldots,X_t^{(d)})$, and $\dd \langle B^{(k)}, B^{(\ell)} \rangle_t = \delta_{k,\ell} \dd t$ by independence of the Brownian components.
\end{theorem}

\begin{proof}
Let us consider the simpler case where $\Phi(x,t) = \Phi(x)$. Using a Taylor-Lagrange expansion on an arbitrarily fine partition $0 = t_0 < t_1 < \dots < t_p = t$, we write the telescopic sum:
\begin{align*}
\Phi(X_t) - \Phi(X_0)
&= \sum_{j=1}^{p} \big(\Phi(X_{t_j})-\Phi(X_{t_{j-1}})\big) \\
&= \sum_{j=1}^{p} \Phi'(X_{t_{j-1}})(X_{t_j}-X_{t_{j-1}}) + \frac{1}{2} \sum_{j=1}^{p} \Phi''(X_{t^\ast_{j-1}})(X_{t_j}-X_{t_{j-1}})^2,
\end{align*}
where $t^\ast_{j-1} \in [t_{j-1},t_{j}]$. By the definition of the stochastic integral, the first sum converges exactly to $\int_0^t \Phi'(X_s)\dd X_s$. For the second sum, by the uniform continuity of $(X_t)_t$, $\Phi''(X_{t^\ast_{j-1}}) \simeq \Phi''(X_{t_{j-1}})$. The squared increments $(X_{t_j}-X_{t_{j-1}})^2$ behave as the quadratic variation, yielding the limit $\int_0^t \Phi''(X_s)\dd \langle X\rangle_s$.
\end{proof}

\section{Diffusion from a distribution and back}

Historically, continuous-time generative models emerged from a simple idea: if we iteratively add noise to data over time $\tau \in [0,T]$, we eventually destroy all structure, reaching a simple, tractable distribution. If we can mathematically \emph{reverse} this stochastic process, we obtain a generative model mapping noise to data.

To cleanly separate this historical forward process from our final generative models, we denote the "data-destroying" path by $Y_\tau$.

\subsection{Ornstein–Uhlenbeck process}
\label{sec:OU}

\begin{definition}[Ornstein-Uhlenbeck process]
\label{def:OU-1D}
An Ornstein-Uhlenbeck (OU) process with parameters $\lambda,\sigma>0$ starting at $y \in \R^d$ is defined by:
\begin{align*}
\begin{cases}
\dd Y_\tau = -\lambda Y_\tau \dd \tau + \sqrt{2} \sigma \dd B_\tau, \\
Y_0 = y.
\end{cases}
\end{align*}
\end{definition}

To solve this SDE, we introduce $U_\tau := e^{\lambda \tau} Y_\tau$. Applying Itô's formula to $\Phi(y,\tau) := e^{\lambda \tau} y$ yields $\dd U_\tau = \sqrt{2} \sigma e^{\lambda \tau} \dd B_\tau$. Integrating this against Brownian motion produces a Gaussian process. 

\begin{figure}
\centering
\includegraphics[width=1\linewidth]{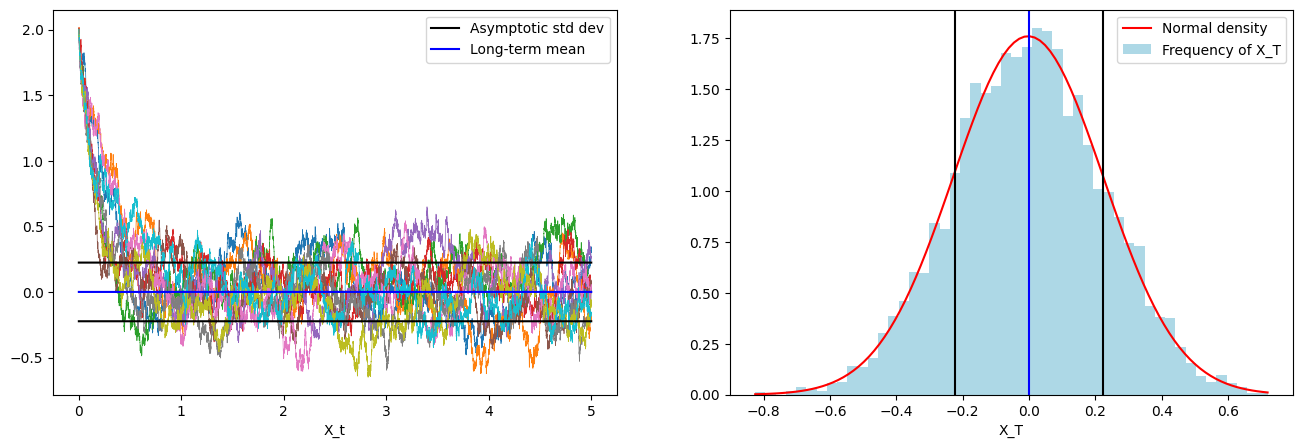}
\caption{Ten trajectories of a homogeneous Ornstein-Uhlenbeck (Definition~\ref{def:OU-1D}) starting from $Y_0 = 2$ with $\lambda = 5$ and $\sigma = 1/2$, all stopped at time $T=5$ (left).
Histogram of $Y_T$ on $N=5000$ draws compared to the limiting normal (right).
}
\label{fig:OU}
\end{figure}

See Figure~\ref{fig:OU} for an illustration.

\begin{proposition}[Time-inhomogeneous Ornstein-Uhlenbeck]
\label{prop:generalized-OU}
The generalized equation $\dd Y_\tau = -f_\tau Y_\tau \dd \tau + \sqrt{2} \sigma_\tau \dd B_\tau$ admits a unique solution. If $Y_0 \sim p^\star(y) \dd y$, then $Y_\tau$ is distributed as
\begin{align*}
Y_0 e^{-\mu_\tau} + \sqrt{\int_0^\tau 2 \sigma_s^2 e^{2(\mu_s-\mu_\tau)} \dd s} \;\; \xi,
\end{align*}
where $\xi \sim \mathcal{N}(0,\mathrm{Id})$ is independent of $Y_0$, and $\mu_\tau := \int_0^\tau f_s \dd s$.
\end{proposition}

\begin{figure}
\centering
\includegraphics[width=0.8\linewidth]{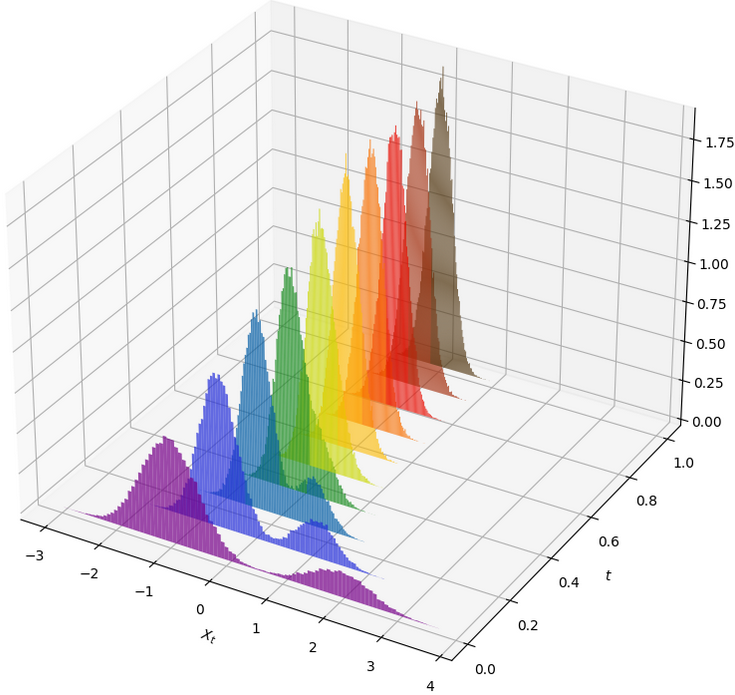}
\caption{
Exemplifying Proposition~\ref{prop:generalized-OU} with histograms of Ornstein-Uhlenbeck processes stopped at $T=1$ starting from $Y_0$ with mixture distribution $p^\star = 0.8 \mathcal{N}(-1,1/2) + 0.2 \mathcal{N}(-2,1/2)$. Diffusion parameters are as in Figure~\ref{fig:OU}. Histograms are computed over $N = 50000$ trajectories.
}
\label{fig:OU-interpolation}
\end{figure}

See Figure~\ref{fig:OU-interpolation} for an illustration of Proposition~\ref{prop:generalized-OU}.

By simulating this OU process forward over a sufficiently large time $T$, the initial data $Y_0$ decays exponentially to zero, and the variance stabilizes such that $Y_T \approx \mathcal{N}(0, \frac{\sigma^2}{\lambda}\mathrm{Id})$. This characterizes our forward process mapping $p^\star$ to pure noise. The fundamental premise of diffusion models is that if we can mathematically reverse this SDE, we can start from pure noise and simulate a trajectory that will exactly terminate at a sample from $p^\star$.

\subsection{Fokker-Planck equation}

To formally reverse this process, we must understand how the probability density of the process evolves. We use PDEs to characterize this evolution.

\begin{proposition}[Fokker-Planck characterization]
\label{prop:fokker-planck}
Let $(Y_\tau)_\tau$ be the solution of $\dd Y_\tau = f_\tau(Y_\tau) \dd \tau + \sqrt{2} \sigma_\tau(Y_\tau) \dd B_\tau$ with initial condition $Y_0 \sim p_0(y) \dd y$ having a smooth density. Then $Y_\tau$ has a density $p_\tau$ satisfying the \emph{Fokker-Planck} equation:
\begin{align*}
\partial_\tau p_\tau(y) = - \nabla \cdot \bigl(f_\tau(y) p_\tau(y) \bigr) + \Delta \bigl(\sigma_\tau^2(y) p_\tau(y)\bigr).
\end{align*}
\end{proposition}

\begin{proof}
Let $\Phi(y)$ be an arbitrary smooth test function with compact support in $\R^d$. From Theorem~\ref{thm:ito-multidimension}, the differential of $\Phi(Y_\tau)$ is:
\begin{align*}
\dd \Phi(Y_\tau)
&= \sum_{k=1}^d \partial_{y_k} \Phi(Y_\tau) \dd Y_\tau^{(k)} + \frac{1}{2} \sum_{k,\ell=1}^d \partial^2_{y_k,y_\ell} \Phi(Y_\tau) \dd \langle Y^{(k)},Y^{(\ell)} \rangle_\tau \\
&= \langle \nabla \Phi(Y_\tau) , \dd Y_\tau \rangle + \sigma_\tau^2(Y_\tau) \Delta \Phi(Y_\tau) \dd \tau,
\end{align*}
where we used that the diffusion matrix is purely diagonal, so $\dd \langle Y^{(k)},Y^{(\ell)} \rangle_\tau = 2\sigma_\tau^2(Y_\tau) \delta_{k,\ell} \dd \tau$. Substituting $\dd Y_\tau = f_\tau(Y_\tau) \dd \tau + \sqrt{2} \sigma_\tau(Y_\tau) \dd B_\tau$ yields:
\begin{align*}
\dd \Phi(Y_\tau)
&= \bigl( \langle \nabla \Phi(Y_\tau) ,  f_\tau(Y_\tau) \rangle + \sigma_\tau^2(Y_\tau) \Delta \Phi(Y_\tau) \bigr) \dd \tau + \sqrt{2} \sigma_\tau(Y_\tau) \langle \nabla \Phi(Y_\tau) , \dd B_\tau \rangle.
\end{align*}
We now take the expectation of both sides. Because the Itô integral with respect to Brownian motion is a zero-mean martingale, the $\dd B_\tau$ term vanishes. We divide by $\dd \tau$ to obtain the time derivative:
\begin{align*}
\frac{\dd}{\dd \tau} \E\bigl[ \Phi(Y_\tau) \bigr]
&= \E\bigl[ \langle \nabla \Phi(Y_\tau) ,  f_\tau(Y_\tau) \rangle + \sigma_\tau^2(Y_\tau) \Delta \Phi(Y_\tau) \bigr].
\end{align*}
Writing the expectations explicitly as spatial integrals against the density $p_\tau(y) \dd y$:
\begin{align*}
\int_{\R^d} \Phi(y) \partial_\tau p_\tau(y) \dd y 
&= \int_{\R^d} \bigl( \langle \nabla \Phi(y) ,  f_\tau(y) \rangle + \sigma_\tau^2(y) \Delta \Phi(y) \bigr) p_\tau(y) \dd y.
\end{align*}
We now apply integration by parts to transfer the spatial derivatives from the test function $\Phi$ to the density terms. Assuming $p_\tau$ and its derivatives vanish at infinity:
\begin{align*}
\int_{\R^d} \langle \nabla \Phi(y), f_\tau(y) \rangle p_\tau(y) \dd y &= - \int_{\R^d} \Phi(y) \nabla \cdot \bigl(f_\tau(y) p_\tau(y)\bigr) \dd y, \\
\int_{\R^d} \Delta \Phi(y) \bigl(\sigma_\tau^2(y) p_\tau(y)\bigr) \dd y &= \int_{\R^d} \Phi(y) \Delta\bigl(\sigma_\tau^2(y) p_\tau(y)\bigr) \dd y \qquad \text{(applying I.B.P. twice)}.
\end{align*}
Substituting these into our integral equation gives:
\begin{align*}
\int_{\R^d} \Phi(y) \partial_\tau p_\tau(y) \dd y = \int_{\R^d} \Phi(y) \Bigl( - \nabla \cdot \bigl(f_\tau(y) p_\tau(y)\bigr) + \Delta\bigl(\sigma_\tau^2(y) p_\tau(y)\bigr) \Bigr) \dd y.
\end{align*}
Because this equality holds for \emph{any} smooth test function $\Phi$, the integrands must be identical almost everywhere, concluding the proof.
\end{proof}

The Fokker–Planck equation can be recast as describing the pure transport of mass corresponding to a deterministic ODE by factoring the Laplacian term.

\begin{proposition}
\label{prop:fokker-planck-transport}
The Fokker-Planck equation for $\dd Y_\tau = f_\tau(Y_\tau) \dd \tau + \sqrt{2} \sigma_\tau(Y_\tau) \dd B_\tau$ can be recast as the non-linear \emph{transport equation}:
\begin{align*}
\partial_\tau p_\tau(y) = - \nabla \cdot \bigl(u_\tau(y) p_\tau(y)\bigr) \quad \text{with velocity field~~~} u_\tau := f_\tau - \sigma_\tau^2 \nabla \log p_\tau - \nabla \sigma_\tau^2.
\end{align*}
Consequently, the trajectories of the deterministic ODE $\dd y_\tau = u_\tau(y_\tau) \dd \tau$ with $y_0 \sim p_0$ share the exact same marginals $p_\tau$.
\end{proposition}
\begin{proof}
Starting from the Fokker-Planck equation:
\begin{align*}
\partial_\tau p_\tau
&= - \nabla \cdot \bigl(f_\tau p_\tau \bigr) + \Delta \bigl(\sigma_\tau^2 p_\tau \bigr) 
= - \nabla \cdot \bigl(f_\tau p_\tau - \nabla \bigl(\sigma_\tau^2 p_\tau \bigr) \bigr) \\
&= - \nabla \cdot \bigl(\{ f_\tau - \nabla \bigl(\sigma_\tau^2 p_\tau \bigr)/p_\tau \} p_\tau \bigr).
\end{align*}
The proof follows by noticing that $\nabla \bigl(\sigma_\tau^2 p_\tau \bigr)/p_\tau = \sigma_\tau^2 \nabla p_\tau /p_\tau + \nabla \sigma_\tau^2 = \sigma_\tau^2 \nabla \log p_\tau + \nabla \sigma_\tau^2$.
\end{proof}

\subsection{Backward processes and time reparametrization}

\subsubsection{Stochastic time-reversal (Anderson's theorem)}
Given a user-defined backward noise schedule $(b_\tau)_{0 \leq \tau \leq T}$, one can reinterpret the Fokker-Planck equation to explicitly derive the backward SDE mapping noise back to data (a consequence of Anderson's theorem).

\begin{theorem}[Backward stochastic dynamic]
\label{thm:backward-SDE}
If the solution to $\dd Y_\tau = f_\tau(Y_\tau) \dd \tau + \sqrt{2} \sigma_\tau(Y_\tau) \dd B_\tau$ has density $p_\tau(y)$, then the reversed-time solution to
\begin{align*}
\begin{cases}
\dd \overleftarrow{Y}_\tau = \overleftarrow{f}_\tau(\overleftarrow{Y}_\tau) \dd \tau + \sqrt{2} b_\tau(\overleftarrow{Y}_\tau) \dd B_\tau \\
\overleftarrow{Y}_0 \sim p_T(y) \dd y
\end{cases}
\end{align*}
with $\overleftarrow{f}_\tau := -f_{T-\tau} + \nabla ( \sigma_{T-\tau}^2 + b_\tau^2) + (\sigma_{T-\tau}^2 + b_\tau^2) \nabla \log p_{T-\tau}$ satisfies $\overleftarrow{Y}_\tau \sim p_{T-\tau}$.
\end{theorem}

\begin{proof}
From Proposition~\ref{prop:fokker-planck}, the Fokker-Planck equation associated with the forward process is
\begin{align*}
0 = \partial_\tau p_\tau + \nabla \cdot \bigl(f_\tau p_\tau \bigr) - \Delta ( \sigma_\tau^2 p_\tau).
\end{align*}
In distribution, reversing the dynamic amounts to considering $\tau \mapsto p_{T-\tau}$ instead of $\tau \mapsto p_{\tau}$, which flips the sign of the time derivative and leaves the spatial ones unchanged:
\begin{align*}
0 = - \partial_\tau p_{T-\tau} + \nabla \cdot \bigl(f_{T-\tau} p_{T-\tau} \bigr) - \Delta ( \sigma_{T-\tau}^2 p_{T-\tau} ).
\end{align*}
To recognize an instance of the Fokker-Planck equation with a new chosen diffusive term $b_\tau^2$, we cleverly add and subtract it:
\begin{align*}
0 &= - \partial_\tau p_{T-\tau} + \left( \nabla \cdot \bigl(f_{T-\tau} p_{T-\tau} \bigr) - \Delta \bigl( (\sigma_{T-\tau}^2+b_\tau^2) p_{T-\tau} \bigr) \right) + \Delta ( b_\tau^2 p_{T-\tau} ).
\end{align*}
Using the identity $\Delta (\sigma^2 p) = \nabla \cdot \bigl( (\sigma^2 \nabla \log p + \nabla \sigma^2) p \bigr)$, the middle term becomes:
\begin{align*}
0 &= - \partial_\tau p_{T-\tau} - \nabla \cdot \bigl(\overleftarrow{f}_\tau p_{T-\tau} \bigr) + \Delta ( b_\tau^2 p_{T-\tau} ),
\end{align*}
where $\overleftarrow{f}_\tau := -f_{T-\tau} + \nabla ( \sigma_{T-\tau}^2 + b_\tau^2) + (\sigma_{T-\tau}^2 + b_\tau^2) \nabla \log p_{T-\tau}$. We recognize exactly the Fokker-Planck PDE characterizing the announced backward SDE.
\end{proof}

\subsubsection{Time-reversal of the O.U. process and reparametrization to $[0,1]$}

Applying Theorem~\ref{thm:backward-SDE} to our Ornstein-Uhlenbeck process where $f_\tau(y) = -\lambda y$ and $\sigma_\tau(y) = \sigma$ (spatially constant), and selecting the symmetric backward diffusion $b_\tau = \sigma$, we obtain the explicit backward generative SDE:
\begin{align*}
\dd \overleftarrow{Y}_\tau = \Big[ \lambda \overleftarrow{Y}_\tau + 2\sigma^2 \nabla \log p_{T-\tau}(\overleftarrow{Y}_\tau) \Big] \dd \tau + \sqrt{2}\sigma \dd B_\tau.
\end{align*}
This equation explicitly tells us how to build a sampling method to generate data: start from pure noise $\overleftarrow{Y}_0 \sim p_T \approx \mathcal{N}(0, \frac{\sigma^2}{\lambda}\mathrm{Id})$ and simulate this SDE over $\tau \in [0,T]$ using a numerical solver. The only unknown term required to run the simulation is the \emph{score} function $\nabla \log p_{T-\tau}(y)$.

However, tracking a ``forward process'' on $\tau \in [0,T]$ and a separate backward process tracking $T-\tau$ is notationally cumbersome and prone to implementation errors. To unify this with modern mathematical approaches, we globally reparametrize the time interval to $t \in [0,1]$ via the linear mapping $t = 1 - \tau/T$. Thus $t=0$ corresponds to pure noise ($\tau=T$ in the backward process), and $t=1$ corresponds to the data ($\tau=0$). 

The generative process $X_t := \overleftarrow{Y}_{tT}$ then runs causally forward from $t=0$ to $t=1$, completely eliminating the need for backward-time PDE tracking. Applying the chain rule $\dd \tau = T \dd t$, and Brownian scaling $\dd B_{\tau} = \dd B_{tT} = \sqrt{T} \dd \tilde{B}_t$ where $\tilde{B}_t$ is a standard Brownian motion on $[0,1]$, we obtain the generative SDE natively parameterized on $t \in [0,1]$:
\begin{equation}
\label{eq:reparametrized-SGM}
\dd X_t = T \Big[ \lambda X_t + 2\sigma^2 \nabla \log p_{T(1-t)}(X_t) \Big] \dd t + \sqrt{2T}\sigma \dd \tilde{B}_t.
\end{equation}
This allows us to discard the historical backward-time formulation entirely and brings us directly to our unified forward-time generative framework.

\section{Building target interpolating dynamics}
\label{sec:common-setup}

We now replace the initial data-destroying processes and build the generative process $X_t$ directly for time $t \in [0,1]$. We start from a tractable base distribution $p_0 = \mathcal{N}(0,\mathrm{Id})$ at $t=0$, and aim to reach the unknown data distribution $p_1 = p^\star$ at $t=1$. Our goal is to find a velocity field $v_t$ that transports mass from $p_0$ to $p_1$ through a causal generative ODE or SDE.

\subsection{Flow matching}

\textbf{1. Interpolating path.} A natural strategy to build such a dynamic is to construct an explicit \emph{kinematic interpolating path} $X^\circ_t = I_t(Z^\circ)$ where $Z^\circ \sim \pi$ is a latent variable coupling the noise and the data on a static probability space. For instance, we may choose $Z^\circ = (\xi^\circ, X_1^\circ) \sim p_0 \otimes p^\star$ and define the path $I_t(Z^\circ) := (1-t)\xi^\circ + t X_1^\circ$. Let $p_t$ be the marginal distribution of $X^\circ_t$. 

\textbf{2. Velocity field.} For any smooth test function $\phi: \mathbb{R}^d \to \mathbb{R}$, the chain rule yields
\begin{align*}
\partial_t \mathbb{E}\big[\phi(X^\circ_t)\big]
&= \mathbb{E}\big[\langle \nabla \phi(I_t(Z^\circ)), \dot{I}_t(Z^\circ) \rangle\big] \\
&= \mathbb{E}\Big[\mathbb{E}\big[\langle \nabla \phi(I_t(Z^\circ)), \dot{I}_t(Z^\circ) \rangle \mid X^\circ_t\big]\Big] \\
&= \mathbb{E}\Big[\big\langle \nabla \phi(X^\circ_t), \mathbb{E}[\dot{I}_t(Z^\circ) \mid X^\circ_t] \big\rangle\Big].
\end{align*}
On the other hand, if we assume that $p_t$ satisfies a continuity equation $\partial_t p_t + \nabla \cdot (v_t p_t) = 0$, integration by parts gives $\partial_t \mathbb{E}[\phi(X^\circ_t)] = \mathbb{E}[\langle \nabla \phi(X^\circ_t), v_t(X^\circ_t) \rangle]$. 
Identifying the two expressions, we see that one valid velocity field is exactly the conditional expectation:
\begin{align*}
v_t(x) = \mathbb{E}\big[\dot{I}_t(Z^\circ) \mid X^\circ_t = x\big].
\end{align*}
Thus, the deterministic ODE $\dd X_t = v_t(X_t) \dd t$ exactly simulates this flow of marginals from pure noise to pure data.

\textbf{3. Population variational problem.} Because $v_t(x)$ is defined as a conditional expectation, it is the unique minimizer of the $L^2$ projection problem at the population level:
\begin{align*}
v_t = \argmin_u \mathbb{E} \Big[ \big\| u(X^\circ_t) - \dot{I}_t(Z^\circ) \big\|^2 \Big].
\end{align*}

\textbf{4. Empirical approximation.} We can approximate this ideal vector field by learning a parametric neural network $\hat{v}_t$ from an empirical dataset. By construction, the latent variables $Z^{\circ(i)} = (\xi^{\circ(i)}, X_1^{(i)})$ are extremely easy to sample: we simply pair a dataset observation $X_1^{(i)} \sim p^\star$ with an independent noise sample $\xi^{\circ(i)} \sim \mathcal{N}(0, \mathrm{Id})$. The empirical variational problem is then directly optimized via the regression loss:
\begin{align*}
\mathcal{L}(\hat{v}_t) = \frac{1}{n} \sum_{i=1}^n \big\| \hat{v}_t(I_t(Z^{\circ(i)})) - \dot{I}_t(Z^{\circ(i)}) \big\|^2.
\end{align*}

\subsection{Gaussian interpolating paths}

\textbf{1. Interpolating path.} Alternatively, we can construct a path of conditional distributions that are explicitly Gaussian, ensuring they are easily samplable. Let $Z^\circ \sim p^\star$ be the latent target data, and let $\xi^\circ \sim \mathcal{N}(0,\mathrm{Id})$ be an independent base noise. We define the interpolating path as:
\begin{align*}
X^\circ_t = m_t(Z^\circ) + \sigma_t \xi^\circ,
\end{align*}
meaning $p_t(\cdot \mid Z^\circ) = \mathcal{N}(m_t(Z^\circ), \sigma_t^2 \mathrm{Id})$. To interpolate from noise to data, we set $m_0(z)=0, \sigma_0=1$ and $m_1(z)=z, \sigma_1=0$.

\textbf{2. Velocity field.} The exact same calculation via test functions as in the previous section yields a valid ODE conditional velocity field:
\begin{align*}
v_t^{\text{ODE}}(x \mid Z^\circ) = \dot{m}_t(Z^\circ) + \frac{\dot{\sigma}_t}{\sigma_t}(x - m_t(Z^\circ)).
\end{align*}
We can also add a diffusion term $\sqrt{2} b_t \dd B_t$ to construct a generative SDE. To maintain the exact same marginals $p_t$, Proposition~\ref{prop:fokker-planck-transport} dictates that the equivalent deterministic transport velocity of our SDE is $v_t^{\text{trans}} = a_t - b_t^2 \nabla \log p_t$. Equating this to $v_t^{\text{ODE}}$, the required conditional SDE drift is $a_t(x \mid Z^\circ) = v_t^{\text{ODE}}(x \mid Z^\circ) + b_t^2 \nabla \log p_t(x \mid Z^\circ)$.

By applying the score trick $\Delta p = \nabla \cdot (p \nabla \log p)$ to the known conditional Gaussian, the conditional score is trivially $\nabla \log p_t(x \mid Z^\circ) = - \frac{x - m_t(Z^\circ)}{\sigma_t^2}$. Gathering these terms, the unified conditional drift mapping noise to data becomes:
\begin{align*}
a_t(x \mid Z^\circ) = \dot{m}_t(Z^\circ) + \left(\frac{\dot{\sigma}_t}{\sigma_t} - \frac{b_t^2}{\sigma_t^2}\right)\big(x - m_t(Z^\circ)\big).
\end{align*}
Marginalizing over $Z^\circ$ yields the exact target drift for the causal generative SDE $\dd X_t = a_t(X_t) \dd t + \sqrt{2} b_t \dd B_t$:
\begin{equation}\label{eq:exact-forward-drift}
a_t(x) = \mathbb{E}\big[a_t(X^\circ_t \mid Z^\circ) \mid X^\circ_t = x\big].
\end{equation}

\textbf{3. Population variational problem.} Because the exact drift is a conditional expectation, it solves the population $L^2$ regression problem:
\begin{align*}
a_t = \argmin_{u} \mathbb{E}_{Z^\circ, \xi^\circ} \Big[ \big\| u(X^\circ_t) - a_t(X^\circ_t \mid Z^\circ) \big\|^2 \Big].
\end{align*}

\textbf{4. Empirical approximation.} Since our data $Z^{\circ(i)} \sim p^\star$ and the base noise $\xi^{\circ(i)} \sim \mathcal{N}(0,\mathrm{Id})$ are easily samplable by construction, we form the empirical samples $X^{\circ(i)}_t = m_t(Z^{\circ(i)}) + \sigma_t \xi^{\circ(i)}$ and compute the exact conditional drift targets $a_t(X_t^{\circ(i)} \mid Z^{\circ(i)})$. We learn the neural network drift $\hat{a}_t$ via the explicit empirical loss:
\begin{align*}
\mathcal{L}(\hat{a}_t) = \frac{1}{n} \sum_{i=1}^n \mathbb{E}_{\xi^\circ} \left[ \big\| \hat{a}_t\big(X^{\circ(i)}_t\big) - a_t\big(X^{\circ(i)}_t \mid Z^{\circ(i)}\big) \big\|^2 \right].
\end{align*}

\subsection{Recovering Score-based generative models (SGM)}

\textbf{1. Interpolating path.} Under this new $[0,1]$ forward-time paradigm, the classic OU-based SGM simply corresponds to selecting the specific Gaussian mean and variance schedules:
\begin{align*}
m_t(Z^\circ) = e^{-\lambda T(1-t)} Z^\circ, \qquad \sigma_t^2 = \frac{\sigma^2}{\lambda}\left(1 - e^{-2\lambda T(1-t)}\right).
\end{align*}
For $\lambda T \gg 1$, the initial distribution $p_0$ is approximately $\mathcal{N}(0, \frac{\sigma^2}{\lambda}\mathrm{Id})$. Because the conditional distribution is explicitly Gaussian, the exact marginal probability density formula is simply:
\begin{align*}
p_t(x) = \int_{\R^d} \cN\left(x; e^{-\lambda T(1-t)} z, \frac{\sigma^2}{\lambda}\left(1 - e^{-2\lambda T(1-t)}\right) \mathrm{Id}\right) p^\star(z) \dd z.
\end{align*}

\textbf{2. Velocity field.} By choosing any generative diffusion $b_t \geq 0$ and simulating the SDE $\dd X_t = a_t(X_t) \dd t + \sqrt{2} b_t \dd B_t$ forward from $t=0$ to $t=1$, we recover the time-reversed OU sampler \eqref{eq:reparametrized-SGM} without invoking any backward-time PDEs during implementation. 
Notice that the explicit conditional drift mapping \eqref{eq:exact-forward-drift} can be entirely re-written in terms of the score trick $\nabla \log p_t$. Let $m_t(z) = \alpha_t z$. Using $v_t^{\text{ODE}}(x) = \frac{\dot{\alpha}_t}{\alpha_t} x + (\sigma_t^2 \frac{\dot{\alpha}_t}{\alpha_t} - \sigma_t \dot{\sigma}_t) \nabla \log p_t(x)$, the explicit marginalized SDE drift is:
\begin{align*}
a_t(x) = \frac{\dot{\alpha}_t}{\alpha_t} x + (\sigma_{\text{fwd},t}^2 + b_t^2) \nabla \log p_t(x),
\end{align*}
where $\sigma_{\text{fwd},t}^2 := \sigma_t^2 \frac{\dot{\alpha}_t}{\alpha_t} - \sigma_t \dot{\sigma}_t \geq 0$ represents the rate of variance destruction of the forward path. 

\textbf{3. Population variational problem.} Because $b_t$ and the variance destruction rate are known scalars, the population variational problem shifts entirely to estimating the score:
\begin{align*}
s_t = \argmin_{s} \mathbb{E} \Big[ \big\| s(X_t^\circ) - \nabla \log p_t(X_t^\circ) \big\|^2 \Big].
\end{align*}

\textbf{4. Empirical approximation.} As will be rigorously proven in Section~\ref{sec:denoisingscorematching} through Denoising Score Matching, we can bypass the intractable marginal score $\nabla \log p_t$. Because the pairs $(Z^{\circ(i)}, \xi^{\circ(i)})$ are trivially samplable, we can replace the target with the explicit conditional score $-\xi^{\circ(i)} / \sigma_t$. This yields the highly efficient empirical regression loss:
\begin{align*}
\mathcal{L}(\hat{s}_t) = \frac{1}{n} \sum_{i=1}^n \mathbb{E}_{\xi^\circ} \left[ \left\| \hat{s}_t\big(X^{\circ(i)}_t, t\big) + \frac{\xi^{\circ(i)}}{\sigma_t} \right\|^2 \right].
\end{align*}
Once the score network $\hat{s}_t$ is learned, it is plugged directly back into the analytic formula $\hat{a}_t(x) = \frac{\dot{\alpha}_t}{\alpha_t} x + (\sigma_{\text{fwd},t}^2 + b_t^2) \hat{s}_t(x)$ to simulate the generative SDE.

\section{Score matching}

Let us rigorously present how to estimate the \emph{score} function $(x,t) \mapsto \nabla \log p_t(x)$. The fundamental metric is the \emph{Fisher divergence}:
\begin{align*}
\mathrm{Fisher}(p_t \mid \widehat{p}_t)
:= \int_{\R^d} \Vert \nabla \log p_t(x) - \nabla \log \widehat{p}_t(x) \Vert^2 p_t(x) \dd x.
\end{align*}
At first glance, this loss cannot be trivially estimated from samples because of the dependence on the true, unknown score $\nabla \log p_t(x)$ inside the expectation.

\subsection{Vanilla score matching}

The main trick for score matching dates back to~\cite{hyvarinen2005estimation}.

\begin{proposition}[Vanilla score trick]
\label{prop:vanilla-score-trick}
For any smooth density $p : \mathbb{R}^d \to \mathbb{R}_+$, there exists a constant $c_p \geq 0$ such that for all smooth vector fields $s:\mathbb{R}^d \to \mathbb{R}^d$ decaying sufficiently fast at infinity,
\begin{align*}
\mathbb{E}_{X \sim p} \left[ \Vert \nabla \log p(X) - s(X) \Vert^2 \right]
= c_p + \mathbb{E}_{X \sim p} \left[ 2 \nabla \cdot s(X) + \Vert s(X) \Vert^2 \right].
\end{align*}
\end{proposition}

\begin{proof}
We expand the left-hand side to get
\begin{align*}
\mathbb{E}_{X \sim p} \left[ \Vert \nabla \log p(X) - s(X) \Vert^2 \right]
= \mathbb{E}_{X \sim p} \left[ \Vert \nabla \log p(X) \Vert^2 \right] - 2 \mathbb{E}_{X \sim p} \left[ \langle \nabla \log p(X), s(X) \rangle \right] + \mathbb{E}_{X \sim p} \left[ \Vert s(X) \Vert^2 \right].
\end{align*}
The first term forms the constant $c_p$. The middle term can be integrated by parts: 
\begin{align*}
-2 \int_{\mathbb{R}^d} \langle \nabla \log p(x), s(x) \rangle p(x) \dd x
&= -2 \int_{\mathbb{R}^d} \langle \nabla p(x), s(x) \rangle \dd x \\
&= 2 \int_{\mathbb{R}^d} p(x) \nabla \cdot s(x) \dd x \\
&= 2 \E_{X \sim p} [ \nabla \cdot s(X) ],
\end{align*}
which yields the result.
\end{proof}

From there, one can fit a parametric set of functions $s_\theta$ (typically neural networks) to learn the score via empirical risk minimization. However, evaluating the divergence $\nabla \cdot s_\theta$ in high dimensions is notoriously costly.

\subsection{Denoising score matching}\label{sec:denoisingscorematching}

To bypass the divergence computation entirely, we can take advantage of the \emph{convolutional} structure of the Gaussian interpolating paths~\cite{vincent2011connection}.

\begin{proposition}[Denoising score trick]
\label{prop:denoising-score-trick}
If $X \sim p(x) \dd x$ and $\varepsilon \sim g(\varepsilon) \dd \varepsilon$ are independent, then $X_\varepsilon := X+\varepsilon \sim (p \ast g)(x) \dd x$.
Furthermore, there exists a constant $c'_{p,g}$ such that for all smooth $s : \mathbb{R}^d \to \mathbb{R}^d$:
\begin{align*}
\mathbb{E}_{X_\varepsilon \sim p \ast g} \left[ \Vert \nabla \log (p \ast g)(X_\varepsilon) - s(X_\varepsilon) \Vert^2 \right]
= c'_{p,g} + \mathbb{E}_{(X,\varepsilon) \sim p \otimes g} \left[ \Vert \nabla \log g (\varepsilon) - s(X + \varepsilon) \Vert^2 \right].
\end{align*}
\end{proposition}

\begin{proof}
Applying Proposition~\ref{prop:vanilla-score-trick} to $X_\varepsilon \sim p \ast g$, we have an expectation involving $2 \nabla \cdot s(X + \varepsilon)$. Expanding this integral explicitly and applying Fubini's theorem:
\begin{align*}
2 \int_{\mathbb{R}^d} \biggl( \int_{\mathbb{R}^d} \nabla_y \cdot s(x+y) g(y) \dd y \biggr) p(x) \dd x
&= -2 \int_{\mathbb{R}^d} \biggl( \int_{\mathbb{R}^d} \langle s(x+y), \nabla g(y) \rangle \dd y \biggr) p(x) \dd x \\
&= -2 \int_{\mathbb{R}^d} \biggl( \int_{\mathbb{R}^d} \langle s(x+y), \nabla \log g(y) \rangle g(y) \dd y \biggr) p(x) \dd x \\
&= -2 \mathbb{E}_{(X,\varepsilon) \sim p \otimes g} \left[ \langle \nabla \log g(\varepsilon), s(X+\varepsilon) \rangle \right],
\end{align*}
where we integrated by parts with respect to $y$. We then complete the square:
\begin{align*}
\mathbb{E}\left[ -2 \langle \nabla \log g(\varepsilon), s(X+\varepsilon) \rangle + \Vert s(X+\varepsilon) \Vert^2 \right]
= \mathbb{E}\left[ \Vert \nabla \log g(\varepsilon) - s(X+\varepsilon) \Vert^2 \right] - \mathbb{E}\left[ \Vert \nabla \log g(\varepsilon) \Vert^2 \right].
\end{align*}
The last term depends only on $g$ and is absorbed into the constant $c'_{p,g}$.
\end{proof}

As desired, the expression given by Proposition~\ref{prop:denoising-score-trick} does not involve any derivative of the candidate score $s$. Instead, the derivative falls on the known score $\nabla \log g$ of the explicit noise. This formally verifies the empirical loss introduced in Section~\ref{sec:common-setup} for Score-Based Generative Models.

\subsubsection{Application to Gaussian interpolating paths}
We apply this trick to the density $p_t$ associated with the Gaussian path $X^\circ_t = m_t(Z^\circ) + \sigma_t \xi^\circ$. 

Its marginal distribution can be written as a convolution $p_t = q_t \ast g_t$, where $q_t$ is the scaled distribution of the deterministic center $m_t(Z^\circ)$, and $g_t$ is the noise distribution of $\varepsilon = \sigma_t \xi^\circ \sim \mathcal{N}(0,\sigma_t^2 \mathrm{Id})$. Hence, $\nabla \log g_t(\varepsilon) = -\varepsilon/\sigma_t^2$.

The time-integrated loss minimization simplifies to
\begin{align*}
\theta
&\in \argmin_{\theta} \int_0^1 w(t) \mathbb{E} \left[ \Vert \nabla \log g_t(\sigma_t \xi^\circ) - s_\theta(X^\circ_t, t) \Vert^2 \right] \dd t \\
&= \argmin_{\theta} \int_0^1 w(t) \mathbb{E} \left[ \left\Vert - \frac{\sigma_t \xi^\circ}{\sigma_t^2} - s_\theta\big(m_t(Z^\circ) + \sigma_t \xi^\circ, t\big) \right\Vert^2 \right] \dd t \\
&= \argmin_{\theta} \int_0^1 w(t) \mathbb{E} \left[ \left\Vert \frac{\xi^\circ}{\sigma_t} + s_\theta(X^\circ_t, t) \right\Vert^2 \right] \dd t.
\end{align*}

\begin{remark}[Tweedie's formula: learning the noise]
In the Gaussian case, a fitted score $s_\theta(\cdot,t)$ is trained to directly match the scaled noise $-\xi^\circ/\sigma_t$ injected into the observation $X^\circ_t$. Looking back at Proposition~\ref{prop:denoising-score-trick}, the optimal minimizer is indeed characterized by the conditional expectation $s^\ast(X^\circ_t) = \mathbb{E} [ \nabla \log g_t (\sigma_t \xi^\circ) \mid X^\circ_t ]$.
\end{remark}

\section{Sampling from a learnt score}
\label{sec:sampling-from-given-score}

\subsection{Exact Kullback-Leibler dynamics}
The following result allows us to track the KL divergence between two probability flows driven by generic transport equations.

\begin{proposition}
\label{prop:kl-of-transport}
Let $(p_t)_{0 \leq t \leq 1}$ and $(q_t)_{0 \leq t \leq 1}$ be two families of smooth probability densities on $\R^d$, respectively driven by the transport equations $\partial_t p_t = - \nabla \cdot (u_t p_t)$ and $\partial_t q_t = - \nabla \cdot (\hat{u}_t q_t)$. Then we have
\begin{align*}
\frac{\dd}{\dd t} \mathrm{KL}(p_t \mid q_t)
&= \int_{\R^d} \left\langle u_t(x) - \hat{u}_t(x) , \nabla \log\left(\frac{p_t(x)}{q_t(x)}\right) \right\rangle p_t(x) \dd x.
\end{align*}
\end{proposition}

\begin{proof}
Assuming we may interchange the time derivative and the integral, we expand the derivative:
\begin{align*}
\frac{\dd}{\dd t} \mathrm{KL}(p_t \mid q_t)
&= \int_{\R^d} \partial_t \left( p_t(x) \log p_t(x) - p_t(x) \log q_t(x) \right) \dd x \\
&= \int_{\R^d} \bigl( \partial_t p_t(x) (1 + \log p_t(x)) - \partial_t p_t(x) \log q_t(x) - \frac{p_t(x)}{q_t(x)} \partial_t q_t(x) \bigr) \dd x \\
&= \int_{\R^d} \partial_t p_t(x) \dd x + \int_{\R^d} \partial_t p_t(x) \log \left(\frac{p_t(x)}{q_t(x)}\right) \dd x - \int_{\R^d} \frac{p_t(x)}{q_t(x)} \partial_t q_t(x) \dd x.
\end{align*}
The first term is zero because $\int_{\R^d} p_t(x) \dd x = 1$ for all $t$. Because $(p_t)_t$ follows the transport equation $\partial_t p_t = - \nabla \cdot (u_t p_t)$, the second term rewrites as
\begin{align*}
- \int_{\R^d} \nabla \cdot \bigl(u_t(x) p_t(x)\bigr) \log\left(\frac{p_t(x)}{q_t(x)}\right) \dd x 
= \int_{\R^d} \left\langle u_t(x) p_t(x), \nabla \log \left(\frac{p_t(x)}{q_t(x)}\right) \right\rangle \dd x,
\end{align*}
using integration by parts. Similarly, substituting the transport equation for $q_t$, $\partial_t q_t = - \nabla \cdot (\hat{u}_t q_t)$, the third term becomes
\begin{align*}
\int_{\R^d} \frac{p_t(x)}{q_t(x)} \nabla \cdot \bigl(\hat{u}_t(x) q_t(x)\bigr) \dd x 
&= - \int_{\R^d} \left\langle \nabla \left(\frac{p_t(x)}{q_t(x)}\right), \hat{u}_t(x) q_t(x) \right\rangle \dd x \\
&= - \int_{\R^d} \left\langle \frac{p_t(x)}{q_t(x)} \nabla \log \left(\frac{p_t(x)}{q_t(x)}\right), \hat{u}_t(x) q_t(x) \right\rangle \dd x \\
&= - \int_{\R^d} \left\langle \hat{u}_t(x) p_t(x), \nabla \log \left(\frac{p_t(x)}{q_t(x)}\right) \right\rangle \dd x,
\end{align*}
where we used the identity $\nabla(p_t/q_t) = (p_t/q_t)\nabla \log(p_t/q_t)$. Summing these terms completes the proof.
\end{proof}

\begin{corollary}
\label{coro:integrated-kl-of-transport}
In the context of Proposition~\ref{prop:kl-of-transport}, we have
\begin{align*}
\mathrm{KL}(p_1 \mid q_1)
&= \mathrm{KL}(p_0 \mid q_0) + \int_0^1 \int_{\R^d} \left\langle u_t(x)-\hat{u}_t(x), \nabla \log\left(\frac{p_t(x)}{q_t(x)}\right) \right\rangle p_t(x) \dd x \dd t.
\end{align*}
\end{corollary}

\subsection{Application to Denoising Diffusion Probabilistic Models (DDPM)}
\label{sec:flow-matching-stability}

\textbf{\textit{(Forward SDE)}}
Let $p_0 : \R^d \to \R_+$ be some probability distribution of interest. Starting from $Y_0 \sim p_0(y) \dd y$, we run the forward data-destroying SDE $\dd Y_\tau = f_\tau(Y_\tau) \dd \tau + \sqrt{2} \sigma_\tau(Y_\tau) \dd B_\tau$ over the time interval $[0,T]$. We denote by $(p_\tau)_{0 \leq \tau \leq T}$ the associated distribution.

\noindent
\textbf{\textit{(True backward (S)DE)}}
Given some backwards noise schedule $(b_\tau)_{0 \leq \tau \leq T}$ (possibly zero), Theorem~\ref{thm:backward-SDE} asserts that the solution to
\begin{align*}
\begin{cases}
\dd \overleftarrow{Y}_\tau = \overleftarrow{f}_\tau(\overleftarrow{Y}_\tau) \dd \tau + \sqrt{2} b_\tau(\overleftarrow{Y}_\tau) \dd B_\tau \\
\overleftarrow{Y}_0 \sim p_T(y) \dd y
\end{cases}
\end{align*}
with $\overleftarrow{f}_\tau := -f_{T-\tau} + \nabla (\sigma_{T-\tau}^2 + b_\tau^2) + (\sigma_{T-\tau}^2 + b_\tau^2) \nabla \log p_{T-\tau}$ satisfies $\overleftarrow{Y}_\tau \sim Y_{T-\tau}$, which we denote by $\overleftarrow{p}_\tau = p_{T-\tau}$.

\noindent
\textbf{\textit{(Approximated backward (S)DE)}}
Given some score function $s : \R^d \times [0,T] \to \R^d$ meant to approximate $\nabla \log p_{T-\tau}$, and an easy-to-sample user-defined density $p_\infty$, we now run the generative SDE:
\begin{align}
\label{eq:final-generator}
\begin{cases}
\dd \hat{Y}_\tau = \hat{f}_\tau(\hat{Y}_\tau) \dd \tau + \sqrt{2} b_\tau(\hat{Y}_\tau) \dd B_\tau \\
\hat{Y}_0 \sim p_\infty(y) \dd y
\end{cases}
\end{align}
with $\hat{f}_\tau := -f_{T-\tau} + \nabla (\sigma_{T-\tau}^2 + b_\tau^2) + (\sigma_{T-\tau}^2 + b_\tau^2) \hat{s}_{T-\tau}$, where $\hat{s}$ is the learn score function that we plug-in from the score matching (or any learning method) step. 
We write $\hat{Y}_\tau \sim \hat{p}_\tau(y) \dd y$. After running the SDE until time $T$, the generated sample output by the method is $\hat{Y}_T$.

\begin{proposition}
\label{prop:kl-bound-for-diffusion-models}
The final time of the stochastic process \eqref{eq:final-generator} satisfies
\begin{align*}
\mathrm{KL}(p_0 \mid \hat{p}_T)
&\leq \mathrm{KL}(p_T \mid p_\infty) + \int_0^T \int_{\R^d} \frac{(\sigma_{T-\tau}^2(y) +  b_{\tau}^2(y))^2}{4 b_{\tau}^2(y)} \left\| \nabla \log p_{T-\tau}(y) - \hat{s}_{T-\tau}(y) \right\|^2 p_{T-\tau}(y) \dd y \dd \tau.
\end{align*}
\end{proposition}

\begin{proof}
From Proposition~\ref{prop:fokker-planck-transport}, its family of densities $\overleftarrow{p}_\tau = p_{T-\tau}$ satisfy the transport equation 
\begin{align*}
\partial_\tau \overleftarrow{p}_\tau &= - \nabla \cdot \bigl(u_\tau \overleftarrow{p}_\tau\bigr) \\
\text{with velocity field~~~}
u_\tau &:= \overleftarrow{f}_\tau - b_\tau^2 \nabla  \log \overleftarrow{p}_\tau - \nabla b_\tau^2 \\
&= - f_{T-\tau} + \sigma_{T-\tau}^2 \nabla \log p_{T-\tau} + \nabla \sigma_{T-\tau}^2.
\end{align*}
Similarly, the family of densities $(\hat{p}_\tau)_{0 \leq \tau \leq T}$ satisfy the transport equation 
\begin{align*}
\partial_\tau \hat{p}_\tau &= - \nabla \cdot \bigl(\hat{u}_\tau \hat{p}_\tau\bigr) \\
\text{with velocity field~~~}
\hat{u}_\tau &:= \hat{f}_\tau - b_\tau^2 \nabla  \log \hat{p}_\tau - \nabla b_\tau^2 \\
&= - f_{T-\tau} + \sigma_{T-\tau}^2 \hat{s}_{T-\tau} - b_\tau^2 \bigl( \nabla  \log \hat{p}_\tau - \hat{s}_{T-\tau} \bigr) + \nabla \sigma_{T-\tau}^2.
\end{align*}
Hence, applying Corollary~\ref{coro:integrated-kl-of-transport} to $(\overleftarrow{p}_\tau)_\tau = (p_{T-\tau})_\tau$ and $(\hat{p}_\tau)_\tau$ on the time interval $[0,T]$ yields 
\begin{align*}
&\mathrm{KL}(p_0 \mid \hat{p}_T) - \mathrm{KL}(p_T \mid p_\infty) \\
&= \int_0^T \int_{\R^d} \left\langle \sigma_{T-\tau}^2 ( \nabla \log p_{T-\tau} - \hat{s}_{T-\tau} ) + b_\tau^2 ( \nabla  \log \hat{p}_\tau - \hat{s}_{T-\tau} ) , \nabla \log\left(\frac{p_{T-\tau}}{\hat{p}_\tau}\right) \right\rangle p_{T-\tau} \dd y \dd \tau \\
&= \int_0^T \int_{\R^d} \left\langle (\sigma_{T-\tau}^2 + b_\tau^2) ( \nabla \log p_{T-\tau} - \hat{s}_{T-\tau} ) - b_\tau^2 \nabla \log\left(\frac{p_{T-\tau}}{\hat{p}_\tau}\right) , \nabla \log\left(\frac{p_{T-\tau}}{\hat{p}_\tau}\right) \right\rangle p_{T-\tau} \dd y \dd \tau.
\end{align*}
In the integrand, the inner product simplifies to
\begin{align*}
- b_\tau^2 \left\| \nabla \log\left(\frac{p_{T-\tau}}{\hat{p}_\tau}\right) \right\|^2 + \left\langle (\sigma_{T-\tau}^2 + b_\tau^2) \bigl( \nabla \log p_{T-\tau} - \hat{s}_{T-\tau} \bigr) , \nabla \log\left(\frac{p_{T-\tau}}{\hat{p}_\tau}\right) \right\rangle \\
&\hspace{3em}
\leq \frac{(\sigma_{T-\tau}^2 +  b_\tau^2)^2}{4 b_\tau^2} \left\| \nabla \log p_{T-\tau} - \hat{s}_{T-\tau} \right\|^2,
\end{align*}
where the inequality follows from Young's inequality $ \left\langle A , B \right\rangle \leq \|A\|^2/(4\lambda) + \lambda \|B\|^2$ with $A = (\sigma_{T-\tau}^2 + b_\tau^2) (\nabla \log p_{T-\tau} - \hat{s}_{T-\tau})$, $B = \nabla \log\left(p_{T-\tau}/\hat{p}_\tau\right)$ and $\lambda = b_\tau^2$. 
The final result follows.
\end{proof}

\bibliographystyle{alpha}
\bibliography{biblio}
\chapter{Generator matching}\label{chap:generatormatching}
\minitoc

Generator Matching provides a common language for a large family of recent generative models.
Rather than trying to learn a distribution $p^\star$ at once, we choose a time-dependent family of intermediate laws $(p_t)_{0\leq t\leq 1}$ linking a simple reference law $p_0$ to $p^\star$.
We then look for a continuous-time Markov process whose marginal law at time $t$ is exactly $p_t$.
Once such a process has been learnt, sampling becomes straightforward: draw $X_0 \sim p_0$, simulate the dynamic process up to time $1$, and output $X_1$.

What makes the framework both general and amenable to training is that we do not try to learn the whole transition kernel of the process.
Instead, we work with its infinitesimal generator.
This viewpoint, recently formalized as \emph{Generator Matching}~\cite{holderrieth2024generator}, simultaneously subsumes flow matching, diffusion models, continuous-time discrete diffusion, and a variety of jump models.
It also enlarges the design space in a concrete way: one may change the probability path, the class of generators, the training loss, or even combine several Markov models together.

As introduced in Chapter~\ref{chap:score-based-generative-models}, we adopt the unified \emph{forward} causal convention on $t \in [0,1]$:
\[
0 \longrightarrow 1,
\qquad
p_0 = \mathcal{N}(0, \mathrm{Id}),
\qquad
p_1 = p^\star,
\]
which natively integrates with the forward generative SDEs and ODEs of the previous chapter. This orientation is convenient for generator matching because it keeps the generative process and the probability path running natively in the same direction.

\section{Probability paths and infinitesimal generators}

\subsection{Probability paths as design objects}

Let $S$ be a measurable state space.
In applications, one may keep in mind $S=\mathbb{R}^d$ for images or Euclidean data, a finite set for discrete data such as text, a manifold for geometric objects, or a product space for multimodal generation.
We are given i.i.d.\ observations from an unknown law $p^\star$ on $S$, and our goal is to generate new samples from that same law.

The first modeling decision is to choose a tractable base distribution $p_0$ from which one can sample easily, together with a \emph{probability path} $(p_t)_{0\leq t\leq 1}$ interpolating between $p_0$ and the target $p^\star$.
As seen in Section~\ref{sec:common-setup}, one usually specifies this path indirectly through a family of \emph{conditional} paths.
One chooses, for each fixed target data point $z^\circ\in S$, a path $\bigl(p_t(\cdot\mid Z^\circ=z^\circ)\bigr)_{0\leq t\leq 1}$ such that $p_0(\cdot\mid Z^\circ=z^\circ)=p_0$ and $p_1(\cdot\mid Z^\circ=z^\circ)=\delta_{z^\circ}$, and one then defines the marginal path by mixing against the latent target $Z^\circ \sim p^\star$:
\[
p_t(A)
=
\int_S p_t(A\mid Z^\circ=z^\circ)\,p^\star(\dd z^\circ),
\qquad A\subseteq S.
\]
Equivalently, if $Z^\circ \sim p^\star$ and, conditionally on $Z^\circ$, we define a static random variable $X^\circ_t\sim p_t(\cdot\mid Z^\circ)$, then the exact marginal law of the kinematic path $X^\circ_t$ is precisely $p_t$.

This conditional viewpoint is central for two reasons.
First, it gives an explicit recipe to manufacture the marginal path.
Second, and more importantly for learning, it is often much easier to compute objects attached to the conditional path than to the marginal one.
Generator Matching exploits this asymmetry systematically.

Let us record two examples that will return throughout the chapter.

\begin{definition}[Two basic conditional probability paths]
\label{def:basic_probability_paths}
Let $Z^\circ \sim p^\star$ be the latent target.

\begin{itemize}[leftmargin=*]
\item
If $S=\mathbb{R}^d$ and we sample an independent base noise $\xi^\circ \sim p_0$, the \emph{geometric-average path} is defined by the simple interpolation
\[
X^\circ_t=(1-\alpha_t)\xi^\circ+\alpha_t Z^\circ,
\qquad 0\leq t\leq 1,
\]
where $(\alpha_t)_{0\leq t\leq 1}$ is a smooth increasing schedule with $\alpha_0=0$ and $\alpha_1=1$.
Its conditional law is
\[
p_t(\dd x\mid Z^\circ=z^\circ)
=
\int_{\mathbb{R}^d}
\delta_{(1-\alpha_t)x_0+\alpha_t z^\circ}(\dd x)\,p_0(\dd x_0).
\]
When $p_0=\mathcal{N}(0,\mathrm{Id})$ and $\alpha_t=t$, this becomes the standard flow matching path:
\[
p_t(\cdot\mid Z^\circ=z^\circ)=\mathcal{N}(t z^\circ,(1-t)^2 \mathrm{Id}).
\]

\item
On an arbitrary state space, the \emph{mixture path} is defined by
\[
p_t(\dd x\mid Z^\circ=z^\circ)
=
(1-\kappa_t)p_0(\dd x)+\kappa_t\,\delta_{z^\circ}(\dd x),
\qquad 0\leq t\leq 1,
\]
where $(\kappa_t)_{0\leq t\leq 1}$ is smooth, non-decreasing, and satisfies $\kappa_0=0$, $\kappa_1=1$.
\end{itemize}
\end{definition}


The geometric-average path moves mass continuously in space.
The mixture path does not: it gradually removes mass from the reference law and reassigns it directly at the endpoint $Z^\circ$.

\subsection{Continuous-time Markov processes and their transition kernels}

Once the static path $(p_t)_t$ has been fixed, the next question is: \emph{which continuous-time dynamics should generate it?}
The relevant objects are continuous-time dynamic Markov processes $(X_t)_{0\leq t\leq 1}$.

\begin{definition}[Continuous-time Markov process]
A stochastic process $(X_t)_{0\leq t\leq 1}$ with values in $S$ is called a \emph{Markov process} if for every measurable set $A\subseteq S$ and every times
\[
0\leq t_1<\cdots<t_n<t_{n+1}\leq 1,
\]
one has
\[
\mathbb{P}\bigl(X_{t_{n+1}}\in A\mid X_{t_1},\ldots,X_{t_n}\bigr)
=
\mathbb{P}\bigl(X_{t_{n+1}}\in A\mid X_{t_n}\bigr).
\]
\end{definition}

Informally, knowing the present makes the past irrelevant for predicting the future.
Associated with such a process is its family of transition kernels
\[
k_{s,t}(x,\dd y),
\qquad 0\leq s\leq t\leq 1,
\]
defined by
\[
\mathbb{P}(X_t\in A\mid X_s=x)=k_{s,t}(x,A).
\]
We shall often write $k_{t+h\mid t}(\dd y\mid x)$ instead of $k_{t,t+h}(x,\dd y)$.

The Markov property can be encoded analytically by the action of these kernels on test functions.
Fix a class $\mathcal{T}$ of bounded measurable functions $f:S\to\mathbb{R}$ rich enough to characterize probability measures.
For $f\in\mathcal{T}$, define
\[
\langle p_t,f\rangle := \int_S f(x)\,p_t(\dd x),
\qquad
k_{t+h\mid t}f(x):=\int_S f(y)\,k_{t+h\mid t}(\dd y\mid x).
\]
Then the tower property yields the basic identity
\[
\langle p_t,k_{t+h\mid t}f\rangle
=
\langle p_{t+h},f\rangle.
\]
This is the weak formulation of marginal propagation.

\subsection{Generator and Kolmogorov forward equation}

The full transition kernel $k_{t+h\mid t}$ is generally far too complicated to parameterize directly.
What is tractable is its first-order behavior for small increments of time.
This leads to the generator.

\begin{definition}[Infinitesimal generator]
Let $(X_t)_{0\leq t\leq 1}$ be a Markov process with transition kernels $(k_{t+h\mid t})$.
For a test function $f\in\mathcal{T}$, assume that the limit
\[
L_t f(x)
:=
\lim_{h\downarrow 0}
\frac{k_{t+h\mid t}f(x)-f(x)}{h}
\]
exists uniformly in $x$.
The operator $L_t$ is called the \emph{generator} of the process at time $t$.
\end{definition}

Thus $L_t$ plays, for the semigroup of transition kernels, the role played by a derivative for a one-parameter family of functions.
It measures the instantaneous evolution of expectations along the process.

\begin{proposition}[Kolmogorov forward equation, weak form]
\label{prop:KFE_weak}
Let $(X_t)_{0\leq t\leq 1}$ be a Markov process with generator $(L_t)_{0\leq t\leq 1}$.
Assume that $t\mapsto \langle p_t,f\rangle$ is differentiable for every $f\in\mathcal{T}$.
Then, for all $f\in\mathcal{T}$,
\[
\frac{\dd}{\dd t}\langle p_t,f\rangle
=
\langle p_t,L_t f\rangle.
\]
\end{proposition}

\begin{proof}
Fix $f\in\mathcal{T}$.
Using the identity $\langle p_t,k_{t+h\mid t}f\rangle=\langle p_{t+h},f\rangle$, we compute
\begin{align*}
\frac{\langle p_{t+h},f\rangle-\langle p_t,f\rangle}{h}
&=
\frac{\langle p_t,k_{t+h\mid t}f-f\rangle}{h}
\\
&=
\left\langle p_t,\frac{k_{t+h\mid t}f-f}{h}\right\rangle.
\end{align*}
By the uniform convergence in the definition of the generator, we may pass to the limit inside the integral and obtain
\[
\frac{\dd}{\dd t}\langle p_t,f\rangle
=
\langle p_t,L_t f\rangle.
\]
\end{proof}

This is the basic equation of the whole framework.
It says that, in order to generate a prescribed path $(p_t)_t$, it is enough to find a generator $L_t$ satisfying the above identity.
Whenever the measures $p_t$ admit densities with respect to a reference measure, one may rewrite the same evolution in the adjoint form $\partial_t p_t = L_t^* p_t$, which recovers the continuity equation for flows, the Fokker--Planck equation for diffusions, and the jump continuity equation for pure-jump processes.

\begin{remark}[The modeling strategy]
The logic of Generator Matching can now be summarized in two steps.
First, choose a kinematic probability path $(p_t)_t$ linking $p_0$ to $p^\star$.
Second, choose a class of Markov generators and solve the weak equation
\[
\frac{\dd}{\dd t}\langle p_t,f\rangle = \langle p_t,L_t f\rangle
\qquad \text{for all }f\in\mathcal{T}.
\]
The rest of the framework is devoted to turning this idea into a scalable learning problem.
\end{remark}

\section{Conditional generators and the matching loss}

\subsection{From conditional paths to the marginal generator}

We now come to the key structural identity behind the whole method.
Suppose that, for every endpoint $Z^\circ$, we have already found a generator $L_t^{Z^\circ}$ whose Markov process generates the conditional path.
Can we reconstruct from them a generator for the marginal path $(p_t)_t$?
The answer is yes, and the formula can be interpreted probabilistically.

\begin{proposition}[Conditional-to-marginal generator formula]
\label{prop:conditional_marginal_generator}
Let $(X^\circ_t)_{t\in[0,1]}$ be a static interpolating process and let
$Z^\circ\sim p^\star$. Assume that, conditionally on $Z^\circ$, the process
$(X^\circ_t)_{t\in[0,1]}$ is Markov with conditional infinitesimal generator
$L_t^{Z^\circ}$. Then, under suitable regularity assumptions, the one-time
marginals of $(X^\circ_t)_{t\in[0,1]}$ evolve according to the operator
\[
L_t f(x)
=
\mathbb{E}\bigl[L_t^{Z^\circ}f(x)\mid X^\circ_t=x\bigr].
\]
Equivalently, for every suitable test function $f$,
\[
\frac{\dd}{\dd t}\mathbb{E}\bigl[f(X^\circ_t)\bigr]
=
\mathbb{E}\bigl[L_t f(X^\circ_t)\bigr].
\]
\end{proposition}

\begin{proof}
For a suitable test function $f$, the tower property gives
\[
\mathbb{E}\bigl[f(X^\circ_t)\bigr]
=
\mathbb{E}\Big[\mathbb{E}\bigl[f(X^\circ_t)\mid Z^\circ\bigr]\Big].
\]
Differentiating under the expectation and using the conditional generator,
\[
\frac{\dd}{\dd t}\mathbb{E}\bigl[f(X^\circ_t)\bigr]
=
\mathbb{E}\Big[
\frac{\dd}{\dd t}
\mathbb{E}\bigl[f(X^\circ_t)\mid Z^\circ\bigr]
\Big]
=
\mathbb{E}\Big[
\mathbb{E}\bigl[L_t^{Z^\circ}f(X^\circ_t)\mid Z^\circ\bigr]
\Big].
\]
Hence
\[
\frac{\dd}{\dd t}\mathbb{E}\bigl[f(X^\circ_t)\bigr]
=
\mathbb{E}\bigl[L_t^{Z^\circ}f(X^\circ_t)\bigr].
\]
Conditioning now on the current state $X^\circ_t$ yields
\[
\mathbb{E}\bigl[L_t^{Z^\circ}f(X^\circ_t)\bigr]
=
\mathbb{E}\Big[
\mathbb{E}\bigl[
L_t^{Z^\circ}f(X^\circ_t)\mid X^\circ_t
\bigr]
\Big].
\]
Thus the desired identity holds with
\[
L_t f(X^\circ_t)
=
\mathbb{E}\bigl[
L_t^{Z^\circ}f(X^\circ_t)\mid X^\circ_t
\bigr],
\]
or, in statewise form,
\[
L_t f(x)
=
\mathbb{E}\bigl[
L_t^{Z^\circ}f(x)\mid X^\circ_t=x
\bigr].
\]
\end{proof}

This proposition is the real reason why conditional paths are useful.
It reduces the construction of a marginal generator to a conditional problem, where the endpoint $Z^\circ$ is known.
In practice, conditional generators are often explicit, while the marginal generator is not.

\begin{figure}[htbp]
    \centering

    \begin{subfigure}{0.9\linewidth}
        \centering
        \includegraphics[width=.9\linewidth]{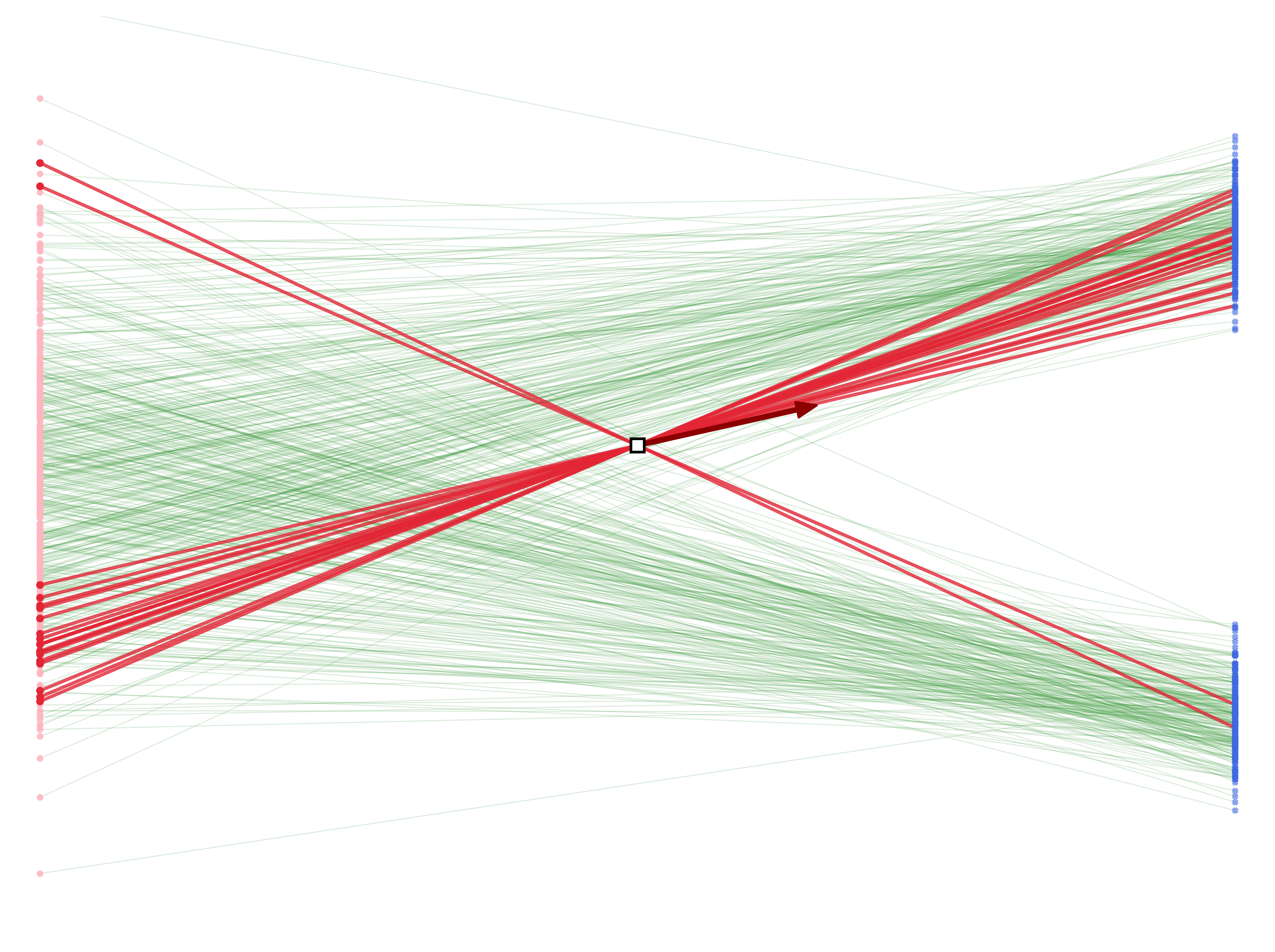}
        \caption{
        The marginal drift field $a_t(x)$, shown as a black arrow, arises as the conditional expectation of the conditional vector fields
        $a_t(x \mid z^\circ)$, shown as red lines, weighted by the posterior probability of each target $z^\circ$ given the current position $x$.
        Here,
        $X_t^\circ = (1-t)\xi^\circ + tX_1^\circ$.
        }
        \label{fig:conditional_vector_field_drift}
    \end{subfigure}

    \vspace{1em}

    \begin{subfigure}{0.9\linewidth}
        \centering
        \includegraphics[width=.9\linewidth]{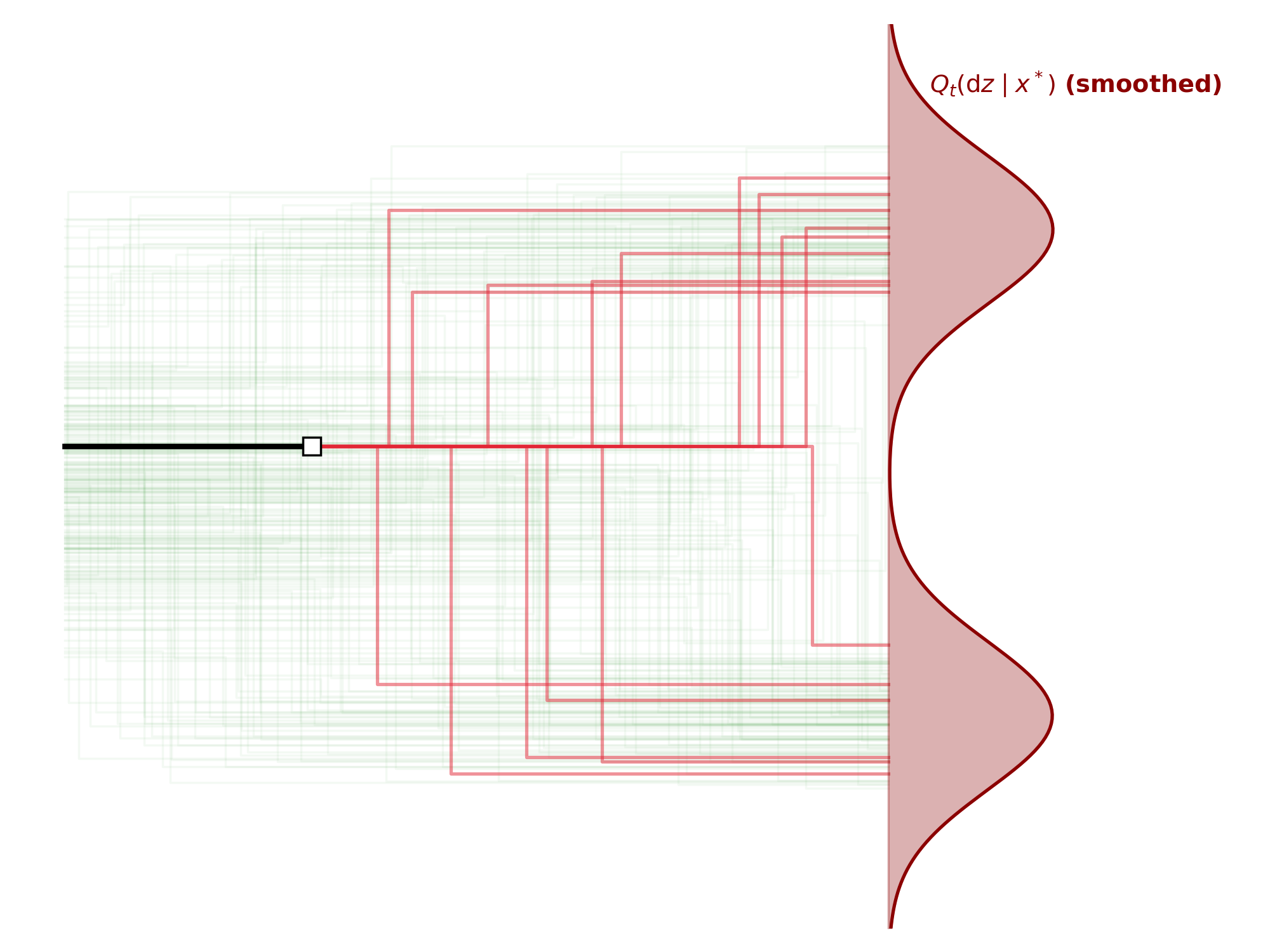}
        \caption{
        Analogous illustration for a pure jump process and the mixture conditional interpolant 
        $p_t(\cdot \mid z^\circ)
        =
        (1-t)\mathcal{U}([-1,1]) + t\delta_{z^\circ}.
        $
        }
        \label{fig:conditional_vector_field_jump}
    \end{subfigure}

    \caption{
    Conditional vector fields and their marginal averages for continuous and jump-process interpolants.
    }
    \label{fig:conditional_vector_field}
\end{figure}

\begin{remark}[Interactive visualizations]
For an excellent intuitive introduction to these concepts---especially why matching conditional paths simplifies the intractable marginal problem visually---we highly recommend the interactive blog post \emph{A Visual Dive into Conditional Flow Matching} \href{https://dl.heeere.com/conditional-flow-matching/}{(link)} by Anne Gagneux \emph{et al.}\ \cite{gagneux2024visual}. It provides visual deep-dives into why conditioning on the target endpoint makes the vector field formulation mathematically tractable and straightforward to simulate. See Figure~\ref{fig:conditional_vector_field} for an illustration of this averaging phenomenon.
\end{remark}

\subsection{Linear parameterizations of generators}

A generator is a linear operator acting on functions.
Neural networks, on the other hand, produce finite-dimensional outputs.
To bridge these two viewpoints, one chooses a linear parameterization of the generator.
The idea is to isolate the dependence on the test function $f$ in a known operator $K$, and to let the neural network predict only the coefficients.

\begin{definition}[Linear parameterization]
\label{def:linear_parametrization}
Let $V$ be a finite-dimensional inner-product space, $\Omega\subseteq V$ a convex set, and $K$ a linear map from test functions to $V$-valued functions on $S$.
A family of generators $(L_t)_t$ is said to admit a \emph{linear parameterization} if there exists a function
\[
F_t:S\to \Omega
\]
such that, for all test functions $f$,
\[
L_t f(x)=\langle Kf(x),F_t(x)\rangle_V.
\]
\end{definition}

This abstract form includes the standard examples.
If $S=\mathbb{R}^d$ and the process is a deterministic transport flow with drift field $a_t$, then
\[
L_t f(x)=\langle a_t(x), \nabla f(x)\rangle,
\]
so one may take $Kf=\nabla f$ and $F_t=a_t$.
For a drift-diffusion SDE process with drift $a_t$ and scalar diffusion parameter $\sqrt{2}b_t$, one has
\[
L_t f(x)=\langle a_t(x), \nabla f(x) \rangle + b_t^2 \Delta f(x),
\]
and one may choose
\[
Kf(x)=\bigl(\nabla f(x), \Delta f(x)\bigr),
\qquad
F_t(x)=\bigl(a_t(x), b_t^2\bigr).
\]
For a jump process,
\[
L_t f(x)=\int_S \bigl(f(y)-f(x)\bigr)Q_t(\dd y;x),
\]
so the predicted object is the jump kernel $Q_t(\cdot;x)$ itself.

\begin{remark}[What the network actually predicts]
In Generator Matching, the neural network does \emph{not} predict the next sample directly.
It predicts the coefficient field $F_t^\theta(x)$ entering the generator.
Sampling is then performed by numerically integrating the Markov dynamics associated with that learned generator.
This separation between \emph{learning} the generator and \emph{simulating} the process is conceptually important.
\end{remark}

\subsection{Bregman divergences and the conditional training objective}

Suppose now that we have fixed a linear parameterization and that $F_t(x)$ denotes the unknown parameter field of the marginal generator, while $F_t^{Z^\circ}(x)$ denotes the corresponding conditional parameter field.
A natural but intractable objective would be to fit $F_t^\theta$ directly to $F_t$ by minimizing
\[
\mathcal L_{\mathrm{GM}}(\theta)
:=
\mathbb E\Bigl[D_\Phi\bigl(F_T(X^\circ_T),F_T^\theta(X^\circ_T)\bigr)\Bigr],
\]
where $T \sim \mathcal{U}(0,1)$ is sampled from a uniform time distribution, and $D_\Phi$ is a Bregman divergence on the parameter space. Note that since the generative dynamic aims to match the marginals of the static path, we evaluate expectations directly over the readily available static samples $X^\circ_T$.
The problem, of course, is that $F_t$ is not known.

The conditional-to-marginal formula suggests the right replacement.
Indeed, the marginal target is the conditional expectation of the conditional target:
\[
F_t(x)=\mathbb E\bigl[F_t^{Z^\circ}(x)\mid X^\circ_t=x\bigr].
\]
This immediately brings to mind a regression problem.
The remarkable fact is that the correct class of losses is precisely the class of Bregman divergences.

\begin{definition}[Bregman divergence]
Let $\Phi: \Omega\to\mathbb R$ be a differentiable strictly convex function on a convex set $\Omega\subseteq V$.
The associated \emph{Bregman divergence} is
\[
D_\Phi(a,b)
:=
\Phi(a)-\Phi(b)-\langle \nabla \Phi(b),a-b\rangle_V.
\]
\end{definition}

The two examples to keep in mind are the squared Euclidean norm, obtained from $\Phi(a)=\|a\|_V^2$, and the Kullback--Leibler divergence on the probability simplex.

\begin{proposition}[Why the conditional loss works]
\label{prop:bregman_conditional_loss}
Let $D_\Phi$ be a Bregman divergence.
Assume that the marginal target is given by
\[
F_t(x)=\mathbb E\bigl[F_t^{Z^\circ}(x)\mid X^\circ_t=x\bigr].
\]
Then the intractable loss
\[
\mathcal L_{\mathrm{GM}}(\theta)
=
\mathbb E\Bigl[D_\Phi\bigl(F_T(X^\circ_T),F_T^\theta(X^\circ_T)\bigr)\Bigr]
\]
has the same gradient in $\theta$ as the tractable conditional loss
\[
\mathcal L_{\mathrm{CGM}}(\theta)
=
\mathbb E\Bigl[D_\Phi\bigl(F_T^{Z^\circ}(X^\circ_T),F_T^\theta(X^\circ_T)\bigr)\Bigr].
\]
In fact,
\[
\mathcal L_{\mathrm{GM}}(\theta)
=
\mathcal L_{\mathrm{CGM}}(\theta)+C,
\]
where $C$ does not depend on $\theta$.
\end{proposition}

\begin{proof}
A Bregman divergence is affine in its first argument.
More precisely, one may rewrite it as
\[
D_\Phi(a,b)=A(a)+\langle a,B(b)\rangle_V + C(b),
\]
where
\[
A(a)=\Phi(a),
\qquad
B(b)=-\nabla\Phi(b),
\qquad
C(b)=-\Phi(b)+\langle \nabla\Phi(b),b\rangle_V.
\]
Hence
\begin{align*}
\mathcal L_{\mathrm{GM}}(\theta)
&=
\mathbb E\Bigl[A(F_T(X^\circ_T))
+\langle F_T(X^\circ_T),B(F_T^\theta(X^\circ_T))\rangle_V
+C(F_T^\theta(X^\circ_T))\Bigr].
\end{align*}
Using the identity $F_t(x)=\mathbb E[F_t^{Z^\circ}(x)\mid X^\circ_t=x]$, we may replace $F_T(X^\circ_T)$ by its conditional expectation in the middle term.
This gives
\begin{align*}
\mathcal L_{\mathrm{GM}}(\theta)
&=
\mathbb E\Bigl[A(F_T(X^\circ_T))\Bigr]
\\
&\qquad +
\mathbb E\Bigl[\langle F_T^{Z^\circ}(X^\circ_T),B(F_T^\theta(X^\circ_T))\rangle_V
+C(F_T^\theta(X^\circ_T))\Bigr].
\end{align*}
Comparing with the same affine decomposition for the conditional loss, we obtain
\[
\mathcal L_{\mathrm{CGM}}(\theta)
=
\mathbb E\Bigl[A(F_T^{Z^\circ}(X^\circ_T))\Bigr]
+
\mathbb E\Bigl[\langle F_T^{Z^\circ}(X^\circ_T),B(F_T^\theta(X^\circ_T))\rangle_V
+C(F_T^\theta(X^\circ_T))\Bigr].
\]
Subtracting the two expressions shows that they differ only by the constant
\[
\mathbb E\bigl[A(F_T(X^\circ_T))-A(F_T^{Z^\circ}(X^\circ_T))\bigr],
\]
which is independent of $\theta$.
The gradients are therefore identical.
\end{proof}

This proposition is the training principle of Generator Matching.
It turns the problem of learning an unknown marginal generator into a supervised learning problem on conditional targets that can be sampled explicitly.
The only ingredients needed in practice are:
\begin{enumerate}
\item a conditional probability path that can be sampled,
\item an explicit conditional generator for that path,
\item a parameterization of the desired generator class,
\item a Bregman divergence.
\end{enumerate}

\begin{remark}[Flow matching as a special case]
If one restricts to flow generators
\[
L_t f(x)=\langle a_t(x) , \nabla f(x) \rangle,
\]
and chooses the squared Euclidean loss, then the conditional Generator Matching objective becomes exactly the conditional flow matching loss.
In that sense, flow matching is not a competing framework: it is the pure-flow subclass of Generator Matching.
\end{remark}

\section{The design space of Markov generators}

\subsection{Zoology of infinitesimal generators}

One of the conceptual strengths of Generator Matching is that it makes the design space explicit.
On Euclidean spaces and on finite discrete spaces, the admissible generators are known rather precisely.
The following statement is the right synthesis for our purposes.

\begin{theorem}[Universal shape of generators]
\label{thm:universal_generators}
Under standard regularity assumptions, the generator of a continuous-time Markov process takes the following form.

\begin{itemize}[leftmargin=*]
\item
If $S$ is finite, then the process is a continuous-time Markov chain and
\[
L_t f(x)=\sum_{y\in S} q_t(y,x) f(y),
\]
where $Q_t=(q_t(y,x))_{x,y\in S}$ is a rate matrix satisfying
\[
q_t(y,x)\geq 0 \quad (y\neq x),
\qquad
\sum_{y\in S} q_t(y,x)=0.
\]

\item
If $S=\mathbb{R}^d$, then
\[
L_t f(x)
=
\langle a_t(x) , \nabla f(x) \rangle
+\frac12 \langle \sigma_t(x) \sigma_t(x)^\top , \nabla^2 f(x)\rangle_{\mathrm{HS}}
+\int_{\mathbb{R}^d}\bigl(f(y)-f(x)\bigr)Q_t(\dd y;x),
\]
where $a_t$ is a drift field, $\sigma_t(x)$ a diffusion matrix, and $Q_t(\cdot;x) = \lambda_{t,x} J_t(\cdot;x)$ is a finite jump measure.
Indeed,
\begin{itemize}[leftmargin=*]
\item If $\dd X_t = a_t(X_t)\dd t$ with $a_t : \R^d \to \R^d$, then a Taylor expansion yields $f(X_{t+h}) = f(X_t) + \langle a_t(X_t), \nabla f(X_t) \rangle + o(h)$. Therefore, after integration $L_tf(x) = \langle a_t(x), \nabla f(x) \rangle$;
\item If $\dd X_t = \sigma_t(X_t)\dd B_t$ with $\sigma_t : \R^d \to \R^{d \times k}$, then the Itô formula yields $f(X_{t+h}) = f(X_t) + \sqrt{h} \langle \sigma_t(X_t) \xi, \nabla f(X_t) \rangle + \frac{h}{2} \langle \sigma_t(X_t) \xi \xi^\top \sigma_t(X_t)^\top, \nabla^2 f(X_t) \rangle_{\mathrm{HS}} + o(h)$. 
Therefore, after integration, $L_tf(x) = \frac{1}{2} \langle \sigma_t(x)\sigma_t(x)^\top, \nabla^2 f(x) \rangle_{\mathrm{HS}}$;
\item If $X_t$ is a jump process with local intensity $\lambda_{t,x}$ and jump distribution $J_t(\mathrm{d}y;x)$, then $X_{t+h} = X_t$ with probability $1-\lambda_{t,X_t} h + o(h)$, and $X_{t+h} = Y \sim J_t(\dd y;X_t)$ with probability $\lambda_{t,X_t} h + o(h)$. Therefore, after integration, $L_t f(x) = +\int_{\mathbb{R}^d}\bigl(f(y)-f(x)\bigr)\lambda_{t,x} J_t(\dd y;x)$.
\end{itemize}
By construction, infinitesimal generators satisfy the conditional minimum principle, meaning that if $x_0 \in \argmin f$, then $L_tf(x_0) \geq 0$. As a result, one can show that every sufficiently regular continuous-time  Markov process on $\R^d$ decomposes into a flow part, a diffusion part, and a jump part (with a locally finite activity in this slightly restrictive case).
This is why generator matching naturally contains flow models, diffusion models, and other currently unexplored pure-jump or mixed models.
\end{itemize}
\end{theorem}

\subsection{Two probability paths, many more generators}

A single probability path may be realized by very different Markov processes.
This is one of the main conceptual differences between a path-based viewpoint and a process-level viewpoint.
Let us work out two basic examples.

\begin{proposition}[Flow solution for the Gaussian conditional path]
\label{prop:flow_solution_condot}
Consider the conditional path on $\mathbb{R}^d$
\[
X^\circ_t=(1-t)\xi^\circ+t Z^\circ,
\qquad \xi^\circ\sim \mathcal N(0,\mathrm{Id}),
\]
so that
\[
p_t(\cdot\mid Z^\circ=z^\circ)=\mathcal N(tz^\circ,(1-t)^2 \mathrm{Id}).
\]
Then this path is generated by the conditional flow field
\[
a_t(x\mid z^\circ)=\frac{z^\circ-x}{1-t},
\qquad 0\leq t<1.
\]
Equivalently, the conditional generator is
\[
L_t^{z^\circ} f(x)= \langle a_t(x\mid z^\circ), \nabla f(x)\rangle.
\]
\end{proposition}

\begin{proof}
By construction,
\[
X^\circ_t=(1-t)\xi^\circ+t Z^\circ.
\]
Differentiating with respect to $t$, we obtain
\[
\dot X^\circ_t=Z^\circ-\xi^\circ.
\]
Since $\xi^\circ=(X^\circ_t-tZ^\circ)/(1-t)$, this can be rewritten as
\[
\dot X^\circ_t
=
Z^\circ-\frac{X^\circ_t-tZ^\circ}{1-t}
=
\frac{Z^\circ-X^\circ_t}{1-t}.
\]
Thus each conditional trajectory solves the ordinary differential equation
\[
\dot X^\circ_t=a_t(X^\circ_t\mid Z^\circ),
\qquad a_t(x\mid z^\circ)=\frac{z^\circ-x}{1-t},
\]
and therefore the corresponding generator is the flow generator stated above.
\end{proof}

This is the basic flow-matching construction.
Observe that the probability path itself did not force us to use a deterministic flow.
It only happened to admit one.
The next example shows the opposite behavior.

\begin{proposition}[Jump solution for the mixture path]
\label{prop:jump_solution_mixture}
Consider the conditional mixture path
\[
p_t(\dd x\mid Z^\circ=z^\circ)
=
(1-\kappa_t)p_0(\dd x)+\kappa_t\,\delta_{z^\circ}(\dd x),
\qquad 0\leq t\leq 1,
\]
with $\kappa_0=0$, $\kappa_1=1$.
Then, for $t<1$, it is generated by the pure-jump operator
\[
L_t^{z^\circ} f(x)
=
\frac{\dot\kappa_t}{1-\kappa_t}\bigl(f(z^\circ)-f(x)\bigr).
\]
Equivalently, the jump intensity is $\dot\kappa_t/(1-\kappa_t)$ and every jump lands at the endpoint $z^\circ$.
\end{proposition}

\begin{proof}
Let $f\in\mathcal T$.
By definition of the path,
\[
\langle p_t(\cdot\mid Z^\circ=z^\circ),f\rangle
=
(1-\kappa_t)\langle p_0,f\rangle+\kappa_t f(z^\circ).
\]
Differentiating in time gives
\[
\frac{\dd}{\dd t}\langle p_t(\cdot\mid Z^\circ=z^\circ),f\rangle
=
\dot\kappa_t\bigl(f(z^\circ)-\langle p_0,f\rangle\bigr).
\]
On the other hand,
\begin{align*}
\langle p_t(\cdot\mid Z^\circ=z^\circ),L_t^{z^\circ} f\rangle
&=
\frac{\dot\kappa_t}{1-\kappa_t}
\int_S \bigl(f(z^\circ)-f(x)\bigr)
\bigl((1-\kappa_t)p_0(\dd x)+\kappa_t\delta_{z^\circ}(\dd x)\bigr)
\\
&=
\dot\kappa_t\bigl(f(z^\circ)-\langle p_0,f\rangle\bigr).
\end{align*}
Thus the weak Kolmogorov forward equation holds.
\end{proof}

This example is particularly instructive.
A mixture path is badly behaved from a pure-flow perspective because mass must disappear from one region and reappear somewhere else.
For a jump generator, this is exactly the native mechanism.
Conversely, the Gaussian path of Proposition~\ref{prop:flow_solution_condot} is especially natural for deterministic transport.
The choice of path and the choice of generator class are therefore tightly coupled.

\begin{remark}[Beyond flows and diffusions]
One of the main contributions of the generator-matching viewpoint is to make jump models in $\mathbb{R}^d$ valid generative models rather than secondary examples.
They are not mere discretizations of diffusions.
They represent genuinely different sample-path geometries, while still solving the same marginal transport problem through the Kolmogorov forward equation.
\end{remark}

\subsection{Combining generators and building multimodal models}

The Kolmogorov forward equation is linear in the generator.
This fact allows us to combine distinct Markov models without changing the marginal probability path.

\begin{proposition}[Linear combinations of generators]
\label{prop:combining_generators}
Let $(p_t)_{0\leq t\leq 1}$ be a probability path.

\begin{itemize}[leftmargin=*]
\item
If $L_t$ and $\widetilde L_t$ both generate $(p_t)_t$, then every convex combination
\[
\bar L_t=\alpha_t L_t+(1-\alpha_t)\widetilde L_t,
\qquad 0\leq \alpha_t\leq 1,
\]
also generates $(p_t)_t$.

\item
If $L_t$ generates $(p_t)_t$ and $L_t^{\circ}$ satisfies
\[
\langle p_t,L_t^{\circ} f\rangle =0
\qquad \text{for all }f\in\mathcal T,
\]
then $L_t+\beta_t L_t^{\circ}$ also generates $(p_t)_t$ for every non-negative scalar function $\beta_t$.
\end{itemize}
\end{proposition}

\begin{proof}
Both claims follow immediately from the linearity of the weak equation
\[
\frac{\dd}{\dd t}\langle p_t,f\rangle = \langle p_t,L_t f\rangle.
\]
For instance, if both $L_t$ and $\widetilde L_t$ satisfy the equation, then
\[
\langle p_t,\bar L_t f\rangle
=
\alpha_t\langle p_t,L_t f\rangle+(1-\alpha_t)\langle p_t,\widetilde L_t f\rangle
=
\frac{\dd}{\dd t}\langle p_t,f\rangle.
\]
The second statement is similar.
\end{proof}

The first construction is a \emph{Markov superposition}.
For example, one may combine a flow generator and a jump generator solving the same equation, and thereby obtain a hybrid process with the same marginals but a different sample-path geometry.
The second construction corresponds to adding a component that is invisible at the level of the marginals.
Predictor--corrector samplers from the diffusion literature fit naturally into this philosophy.

A second use of linearity appears in product spaces.
Suppose that $S=S_1\times S_2$ and that we have conditional paths
\[
p_t^{(1)}(\cdot\mid Z_1^\circ),
\qquad
p_t^{(2)}(\cdot\mid Z_2^\circ)
\]
on the two components.
If the conditional path on the product space factorizes as
\[
p_t\bigl(\dd(x_1,x_2)\mid (Z_1^\circ,Z_2^\circ)\bigr)
=
p_t^{(1)}(\dd x_1\mid Z_1^\circ)\,p_t^{(2)}(\dd x_2\mid Z_2^\circ),
\]
then a conditional generator is obtained simply by adding the componentwise generators:
\[
L_t^{(Z_1^\circ,Z_2^\circ)}f(x_1,x_2)
=
L_t^{1,Z_1^\circ}\bigl(f(\cdot,x_2)\bigr)(x_1)
+
L_t^{2,Z_2^\circ}\bigl(f(x_1,\cdot)\bigr)(x_2).
\]
This is the formal reason why multimodal generator matching is relatively straightforward to set up: one may keep the update mechanism appropriate for each modality, while allowing the parameters of each component to depend on the full multimodal state.

\section{Inference and stability for generator matching}

\subsection{A dual stability identity via the backward Kolmogorov equation}

Assume for now that a certain training strategy has delivered us an approximate generator $\widehat L_t$ rather than the exact one.
To understand how generator errors propagate to the final law, it is natural to turn to the Kolmogorov backward equation.
This dual viewpoint is standard in classical stochastic analysis limits~\cite{kunita1990stochastic}, but here it becomes especially useful because generator errors are exactly what the training loss controls.

\begin{proposition}[Duality identity for generator errors]
\label{prop:duality_generator_error}
Let $(p_t)_{0\leq t\leq 1}$ and $(\widehat p_t)_{0\leq t\leq 1}$ solve the weak Kolmogorov forward equations associated with generators $L_t$ and $\widehat L_t$, with the same initial condition $p_0=\widehat p_0$.
Let $f\in\mathcal T$, and let $(\widehat{f}_t)_{0\leq t\leq 1}$ solve the backward Kolmogorov equation
\[
\partial_t \widehat{f}_t + \widehat L_t\widehat{f}_t =0,
\qquad
\widehat{f}_1=f.
\]
Then
\[
\langle p_1-\widehat p_1,f\rangle
=
\int_0^1 \langle p_t,(L_t-\widehat L_t)\widehat{f}_t\rangle\,\dd t.
\]
\end{proposition}

\begin{proof}
Differentiate the pairing $t\mapsto \langle p_t,\widehat{f}_t\rangle$.
By the forward equation for $p_t$ and the backward equation for $\widehat{f}_t$,
\begin{align*}
\frac{\dd}{\dd t}\langle p_t,\widehat{f}_t\rangle
&=
\langle p_t,L_t\widehat{f}_t\rangle+\langle p_t,\partial_t\widehat{f}_t\rangle
\\
&=
\langle p_t,(L_t-\widehat L_t)\widehat{f}_t\rangle.
\end{align*}
Similarly,
\[
\frac{\dd}{\dd t}\langle \widehat p_t,\widehat{f}_t\rangle
=
\langle \widehat p_t,\widehat L_t\widehat{f}_t+\partial_t\widehat{f}_t\rangle
=
0.
\]
Integrating both identities from $0$ to $1$, and using $p_0=\widehat p_0$ and $\widehat{f}_1=f$, we obtain
\begin{align*}
\langle p_1,f\rangle-\langle \widehat p_1,f\rangle
&=
\langle p_1,\widehat{f}_1\rangle-\langle \widehat p_1,\widehat{f}_1\rangle
\\
&=
\int_0^1 \langle p_t,(L_t-\widehat L_t)\widehat{f}_t\rangle\,\dd t.
\end{align*}
\end{proof}

This formula is fundamental.
It shows that final discrepancies are controlled by two ingredients only:
\begin{enumerate}
\item the size of the generator error $L_t-\widehat L_t$ along the true path,
\item the regularity of the dual solution $\widehat{f}_t$ to the backward equation.
\end{enumerate}
The first item is the estimation problem.
The second is a stability problem.
The next proposition addresses the latter in a simple but useful case.

\begin{proposition}[Backward Lipschitz propagation with additive diffusion]
\label{prop:backward_lipschitz}
Assume that
\[
\widehat L_t \phi(x)=\langle \widehat a_t(x), \nabla \phi(x) \rangle +b_t^2\,\Delta \phi(x),
\]
where $b_t^2\geq 0$ depends only on time, and $\widehat a_t:\mathbb{R}^d\to\mathbb{R}^d$ is of class $C^1$ in space.
Assume moreover that
\[
\lambda_{\max}\bigl(\nabla \widehat a_t(x)\bigr)\leq \ell_t
\qquad \text{for all }x\in\mathbb{R}^d,
\]
for some integrable function $\ell:[0,1]\to\mathbb{R}$.
Let $\widehat{f}_t$ solve
\[
\partial_t \widehat{f}_t + \widehat L_t\widehat{f}_t =0,
\qquad
\widehat{f}_1=f,
\]
with $f\in \mathrm{Lip}(\mathbb{R}^d)$.
Then, for every $t\in[0,1]$,
\[
\|\widehat{f}_t\|_{\mathrm{Lip}}
\leq
\exp\!\left(\int_t^1 \ell_u\,\dd u\right)
\|f\|_{\mathrm{Lip}}.
\]
\end{proposition}

\begin{proof}
Fix $t\in[0,1]$ and $x,y\in\mathbb{R}^d$.
Let $(B_s)_{s\geq t}$ be a Brownian motion, and consider the coupled diffusions
\[
X_s^{t,x}=x+\int_t^s \widehat a_u(X_u^{t,x})\,\dd u+\int_t^s \sqrt{2}b_u\,\dd B_u,
\]
\[
X_s^{t,y}=y+\int_t^s \widehat a_u(X_u^{t,y})\,\dd u+\int_t^s \sqrt{2}b_u\,\dd B_u,
\qquad s\in[t,1],
\]
driven by the same Brownian motion.
By the Feynman--Kac representation for the backward equation,
\[
\widehat{f}_t(x)=\mathbb E\bigl[f(X_1^{t,x})\bigr],
\qquad
\widehat{f}_t(y)=\mathbb E\bigl[f(X_1^{t,y})\bigr].
\]
Therefore,
\[
|\widehat{f}_t(x)-\widehat{f}_t(y)|
\leq
\|f\|_{\mathrm{Lip}}\,\mathbb E\bigl[|X_1^{t,x}-X_1^{t,y}|\bigr].
\]

Set $Z_s:=X_s^{t,x}-X_s^{t,y}$.
Because the two diffusions use the same Brownian motion and the diffusion coefficient is additive, the noise cancels and we obtain
\[
Z_s=x-y+\int_t^s \bigl(\widehat a_u(X_u^{t,x})-\widehat a_u(X_u^{t,y})\bigr)\,\dd u.
\]
Hence
\[
\frac{\dd}{\dd s}|Z_s|^2
=
2\,\langle Z_s, \widehat a_s(X_s^{t,x})-\widehat a_s(X_s^{t,y})\rangle.
\]
By the mean-value theorem, and because $\widehat a_s$ is one-sided Lipschitz with constant $\ell_s$, we have
\[
\langle Z_s, \widehat a_s(X_s^{t,x})-\widehat a_s(X_s^{t,y})\rangle
\leq
\ell_s |Z_s|^2,
\]
so that
\[
\frac{\dd}{\dd s}|Z_s|^2\leq 2\ell_s |Z_s|^2.
\]
Gronwall's lemma yields
\[
|Z_s|^2
\leq
\exp\!\left(2\int_t^s \ell_u\,\dd u\right)|x-y|^2,
\]
and therefore
\[
|X_1^{t,x}-X_1^{t,y}|
\leq
\exp\!\left(\int_t^1 \ell_u\,\dd u\right)|x-y|
\qquad \text{almost surely}.
\]
Substituting into the previous bound gives
\[
|\widehat{f}_t(x)-\widehat{f}_t(y)|
\leq
\|f\|_{\mathrm{Lip}}
\exp\!\left(\int_t^1 \ell_u\,\dd u\right)|x-y|.
\]
Taking the supremum over $x\neq y$ concludes the proof.
\end{proof}

Combined with Proposition~\ref{prop:duality_generator_error}, this estimate shows how one-sided Lipschitz bounds on the learnt drift control the amplification of generator errors.
This is the basic mechanism behind many Wasserstein-type stability estimates for score and flow models. As Proposition~\ref{prop:drift_diff_lip_coupling} shows, the critical analytic ingredient making this propagation possible is bounding $\lambda_{\max}(\nabla \widehat a_t)$ (or its exact counterpart $\lambda_{\max}(\nabla a_t)$). Establishing this bound is the core objective of the next chapter.

\begin{proposition}\label{prop:drift_diff_lip_coupling}
Assume that, for each $t \in [0,1]$,
\[
L_t \phi(x) = \langle a_t(x) \nabla \phi(x) \rangle + b_t^2 \Delta \phi(x),
\qquad
\widehat L_t \phi(x) = \widehat \langle a_t(x), \nabla \phi(x) \rangle + b_t^2 \Delta \phi(x),
\]
and that $(p_t)_{0\le t\le 1}$ and $(\widehat p_t)_{0\le t\le 1}$ solve the corresponding weak
Kolmogorov forward equations, with the same initial condition
\[
p_0 = \widehat p_0.
\]
Assume moreover that $\widehat a_t$ is of class $C^1$ in space and that
\[
\lambda_{\max}\bigl(\nabla \widehat a_t(x)\bigr) \le \ell_t
\qquad \text{for all } x \in \mathbb{R}^d,
\]
for some integrable function $\ell \in L^1([0,1])$.
Then
\[
W_1(p_1,\widehat p_1)
\le
\int_0^1
\exp\!\left(\int_t^1 \ell_u\,\dd u\right)
\|a_t-\widehat a_t\|_{L^1(p_t)}\,\dd t.
\]
\end{proposition}

\begin{proof}
Let $f:\mathbb{R}^d\to\mathbb{R}$ be $1$-Lipschitz, and let $(\phi_t)_{0\le t\le 1}$ solve
\[
\partial_t \phi_t + \widehat L_t \phi_t = 0,
\qquad
\phi_1 = f.
\]
By Proposition \ref{prop:duality_generator_error},
\[
\langle p_1-\widehat p_1,f\rangle
=
\int_0^1 \langle p_t,(L_t-\widehat L_t)\phi_t\rangle\,\dd t.
\]
Since
\[
(L_t-\widehat L_t)\phi_t(x)
=
\langle a_t(x)-\widehat a_t(x), \nabla \phi_t(x)\rangle,
\]
we get
\[
|\langle p_1-\widehat p_1,f\rangle|
\le
\int_0^1
\int_{\mathbb{R}^d}
|a_t(x)-\widehat a_t(x)|\,|\nabla \phi_t(x)|\,p_t(\dd x)\,\dd t.
\]
By Proposition \ref{prop:backward_lipschitz},
\[
\|\phi_t\|_{\mathrm{Lip}}
\le
\exp\!\left(\int_t^1 \ell_u\,\dd u\right)\|f\|_{\mathrm{Lip}}
\le
\exp\!\left(\int_t^1 \ell_u\,\dd u\right),
\]
hence
\[
|\langle p_1-\widehat p_1,f\rangle|
\le
\int_0^1
\exp\!\left(\int_t^1 \ell_u\,\dd u\right)
\|a_t-\widehat a_t\|_{L^1(p_t)}\,\dd t.
\]
Taking the supremum over all $1$-Lipschitz $f$ and using the Kantorovich--Rubinstein duality yields
\[
W_1(p_1,\widehat p_1)
\le
\int_0^1
\exp\!\left(\int_t^1 \ell_u\,\dd u\right)
\|a_t-\widehat a_t\|_{L^1(p_t)}\,\dd t.
\]
\end{proof}

\subsection{Empirical risk minimization and a multiscale statistical viewpoint}

Let us finally return to the learning problem itself.
Given data $Z^{\circ(1)},\ldots,Z^{\circ(n)}\sim p^\star$, one choses times $0 \leq t_k \leq 1$ and conditional states
\[
X^\circ_i\sim p_{t_k}(\cdot\mid Z^\circ=Z^{\circ(i)}),
\]
and minimizes an empirical version of the conditional loss,
\[
\widehat{\mathcal L}_{\mathrm{CGM}}(\theta)
=
\frac1n\sum_{i=1}^n
D_\Phi\bigl(F_{t_k}^{Z^{\circ(i)}}(X^\circ_{t_k}),F_{t_k}^\theta(X^\circ_{t_k})\bigr).
\]
At this level of generality, Generator Matching is just empirical risk minimization for a time-dependent regression problem.
The difficulty is that the regression problem is often highly non-uniform in time.

This non-uniformity is especially visible in score-based diffusion models.
Near the low-noise end of the path, the conditional targets become increasingly singular, the score norm grows, and the associated vector fields may have large Lipschitz constants.
Farther away from the data, the regression problem is statistically easier, but the generated law is then more sensitive to truncation bias.
The handwritten notes accompanying this chapter insist on exactly this point: \emph{one should not expect the same function class, the same resolution, or the same error budget to be optimal at every time scale.}

This leads naturally to three practical ideas.
First, one often truncates the time interval away from the singular endpoint and only learns on $[0,1-\epsilon]$ or, depending on the time convention, on $[\tau,\overline T]$ after a suitable reparameterization.
Second, one may use time weights in the loss in order to compensate for the uneven statistical difficulty across the path.
Third, and most importantly in nonparametric analyses, one decomposes time into dyadic or geometric blocks and allows the complexity of the estimator to vary from one block to the next.

\begin{remark}[Why multiscale time parameterizations appear naturally]
For a smooth data distribution, the smoothed laws $(p_t)_t$ become easier to estimate as noise increases, but the associated generator no longer describes the near-data regime accurately.
Conversely, the low-noise regime is statistically much harder and numerically less stable.
A multiscale architecture, with one level of approximation per time block, is therefore not just an implementation trick.
It reflects the intrinsic non-uniformity of the statistical problem along the probability path.
\end{remark}

At a conceptual level, the proof strategy behind recent minimax analyses of score-based models follows the same pattern as the one suggested by Propositions~\ref{prop:bregman_conditional_loss}, \ref{prop:duality_generator_error}, and \ref{prop:backward_lipschitz}:
\begin{enumerate}
\item control the regression error of the learnt generator on each time block,
\item propagate that error to the final distribution through the backward equation,
\item balance approximation, estimation, and truncation errors across the different time scales.
\end{enumerate}
The precise rates depend on the smoothness class of the data distribution and on the generator class under study, but the general mechanism is robust and extends well beyond classical diffusion models.

\begin{remark}[What the framework buys us]
The point of Generator Matching is not merely to rename known models.
It isolates three linear structures that are otherwise partly hidden:
\begin{itemize}[leftmargin=*]
\item the linearity of the generator in the test function,
\item the linearity of the Kolmogorov forward equation in the generator,
\item the affine dependence of Bregman divergences on the training target.
\end{itemize}
Together, these three facts explain why conditional regression suffices, why generators can be combined, and why multimodal or hybrid Markov models can be designed with very little additional formalism.
\end{remark}

The chapter may therefore be read as follows.
A probability path specifies \emph{what marginal evolution we want}.
A generator specifies \emph{how a Markov process realizes this evolution infinitesimally}.
Generator Matching turns this into a conditional supervised learning problem.
The rest --- flow matching, score-based diffusion, jump models, multimodal generation, stability estimates, and statistical rates --- follows from that single principle.

\bibliographystyle{alpha}
\bibliography{biblio}
\chapter{Statistical Motivations and Score Regularity}\label{chap:scoreregularity}
\minitoc

The goal of this chapter is to bridge the continuous-time generative frameworks constructed in Chapters~\ref{chap:score-based-generative-models} and~\ref{chap:generatormatching} with their statistical estimation from finite data. 

Once a probability path $(p_t)_{0\le t\le 1}$ has been fixed, generative modeling becomes a regression problem: learn the exact drift $a_t$ well enough so that the law of the learned dynamics remains close to the target law. We first expose why the minimax optimal recovery of a target distribution $p^\star$ hinges heavily on two precise regularity properties of the exact generative drift $a_t$: its one-sided Lipschitz constant (which governs dynamic stability) and its localized space-time Hölder smoothness (which governs neural approximation). After establishing this motivation, we formally and pedagogically prove that these regularities are indeed satisfied for standard Gaussian interpolations.

\section{The case of Gaussian interpolation with chosen diffusions}

We operate within the Gaussian interpolation framework introduced in Section~\ref{sec:common-setup}. Thus, the target distribution is a probability density $p^\star$ on $\mathbb{R}^d$, and we choose an explicit kinematic interpolating path $(p_t)_{0\le t\le 1}$ such that there exists a random variable:
\begin{equation}\label{eq:chapter3_reference_path}
X^\circ_t = m_t(Z^\circ)+\sigma_t \xi^\circ, \qquad t \in [0,1], \quad Z^\circ\sim p^\star, \quad \xi^\circ\sim \mathcal N(0,\mathrm{Id}),
\end{equation}
with $Z^\circ$ and $\xi^\circ$ independent, $m_0(z)=0, \sigma_0=1$, and $m_1(z)=z, \sigma_1=0$. The marginal law of $X^\circ_t$ is exactly $p_t$. 

It is clear from this definition that, conditionally on the target, $(X^\circ_t \mid Z^\circ=z) \sim \cN(m_t(z), \sigma_t^2 \mathrm{Id})$. As established in Chapter~\ref{chap:generatormatching}, solving the associated Fokker--Planck equation ensures that the causal generative SDE
\begin{equation}\label{eq:chapter3_exact_sde}
\dd X_t = a_t(X_t)\,\dd t + \sqrt{2}\,b_t\,\dd B_t, \qquad X_0\sim \mathcal N(0,\mathrm{Id}),
\end{equation}
exactly simulates this marginal path, yielding $X_1 \sim p^\star$. In the Gaussian-mixture setting ($\sigma_t>0$), the exact drift takes the explicit conditional expectation form:
\begin{equation}\label{eq:chapter3_exact_drift}
a_t(x)
= \mathbb E\left[\dot{m}_t(Z^\circ) + \left(\frac{\dot{\sigma}_t}{\sigma_t}-\frac{b_t^2}{\sigma_t^2}\right)\sigma_t\xi^\circ \;\middle|\; X^\circ_t=x\right]
= \mathbb E\left[\dot{m}_t(Z^\circ) + \eta_t\big(x-m_t(Z^\circ)\big) \;\middle|\; X^\circ_t=x\right],
\end{equation}
where $\eta_t := \dot{\sigma}_t/\sigma_t - b_t^2/\sigma_t^2$.

Because \eqref{eq:chapter3_exact_drift} is a conditional expectation, $a_t$ is uniquely characterized as the minimizer of a quadratic regression problem at the population level:
\begin{equation}\label{eq:chapter3_population_regression}
a_t \in \argmin_{\phi_t\in L^2(p_t)} \mathbb E\left[\left\|\phi_t(X^\circ_t) - \left(\dot{m}_t(Z^\circ)+\eta_t(X^\circ_t-m_t(Z^\circ))\right)\right\|^2\right].
\end{equation}

Given data $Z^{\circ(1)},\dots,Z^{\circ(n)}\sim p^\star$ and auxiliary base noises $\xi^{\circ(1)},\dots,\xi^{\circ(n)}\sim \mathcal N(0,\mathrm{Id})$, we approximate $a_t$ by empirical risk minimization over a neural network class $\mathcal V_t$:
\begin{equation}\label{eq:chapter3_empirical_regression}
\hat a_t \in \argmin_{\phi_t\in \mathcal V_t} \frac1n\sum_{i=1}^n \left\|\phi_t\bigl(m_t(Z^{\circ(i)})+\sigma_t\xi^{\circ(i)}\bigr) - \left(\dot{m}_t(Z^{\circ(i)})+\eta_t(\sigma_t\xi^{\circ(i)})\right)\right\|^2.
\end{equation}
Once the sequence $(\hat a_t)_{0\le t\le 1}$ is trained, we deploy the learned generator $\dd\hat X_t = \hat a_t(\hat X_t)\,\dd t + \sqrt{2}\,b_t\,\dd B_t$. Our final generative estimator is the terminal law $\hat p_n := \mathrm{Law}(\hat X_{T_n})$, evaluated at a deterministic early-stopping time $T_n \le 1$.

\begin{definition}[Benchmark models]\label{def:chapter3_benchmarks}
We specialize our continuous-time analysis to the two concrete interpolations:
\begin{itemize}[leftmargin=*]
    \item \textbf{Vanilla flow matching:} $X^\circ_t = t Z^\circ + (1-t)\xi^\circ$, meaning $b_t=0$.
    \item \textbf{Rescaled diffusion:} $X^\circ_t = t Z^\circ + \sqrt{1-t^2}\xi^\circ$, meaning $b_t=t^{-1/2}$.
\end{itemize}
\end{definition}

This viewpoint underscores that the statistical proof does not fundamentally distinguish ``flow matching'' from ``diffusion''. It merely distinguishes classes of drifts possessing slightly different analytic profiles for $m_t$, $\sigma_t$, and $b_t$.

\section{What the minimax problem asks for}

\subsection{The nonparametric benchmark rate}
To evaluate the statistical optimality of the learned model, we assume the target distribution $p^\star$ is supported on $[0,1]^d$ and sits in a Hölder ball measuring $\beta$-smoothness:
\[
\mathcal H_K^\beta([0,1]^d) := \left\{ f:[0,1]^d\to\mathbb R_+ \ \middle|\ \sum_{i=0}^{\lfloor \beta\rfloor}\|\nabla^i f\|_\infty + \|\nabla^{\lfloor \beta\rfloor}f\|_{\beta-\lfloor \beta\rfloor} \le K \right\}.
\]
A classic theorem of nonparametric statistics states that the minimax risk to estimate a density in this class, measured in the 1-Wasserstein distance $W_1$, converges at the rate:
\begin{equation}\label{eq:chapter3_minimax_rate}
\mathcal R_n(\mathcal H_K^\beta,W_1) := \inf_{\tilde p_n} \sup_{p^\star\in \mathcal H_K^\beta} \mathbb E\big[W_1(p^\star,\tilde p_n)\big] \asymp n^{-\frac{\beta+1}{2\beta+d}}.
\end{equation}
This exponent is both a lower bound (no estimator can converge faster uniformly on $\mathcal H_K^\beta$) and an upper bound (achievable up to polylogarithmic factors). The core statistical question is therefore: \emph{How accurately must one learn the drift $a_t$ so that the generated law $\hat p_n$ attains this optimal rate?}

\subsection{Why one-sided Lipschitz control matters: Stability}
The mechanism translating generator errors to the final distribution error was isolated via the backward Kolmogorov duality in Proposition~\ref{prop:drift_diff_lip_coupling}. It yields the following stability reduction.

\begin{proposition}[Stability reduction]\label{prop:chapter3_stability_reduction}
Let $p_t$ be the law of the exact dynamics \eqref{eq:chapter3_exact_sde}, and let $\hat p_t$ be the law of the learned dynamics, both started from $\mathcal N(0,\mathrm{Id})$.
Assume that, for every $t\in[0,1]$, the field $\hat a_t$ satisfies the one-sided Lipschitz condition:
\[
\sup_{x\in\mathbb R^d}\lambda_{\max}(\nabla \hat a_t(x))\le \hat\ell_t, \qquad \hat\ell_t\in L^1(0,1).
\]
Then the Wasserstein discrepancy satisfies:
\begin{equation}\label{eq:chapter3_stability_bound}
W_1(p_{T_n},\hat p_n) \le \int_0^{T_n} \exp\!\left(\int_t^{T_n} \hat\ell_s\,\dd s\right) \int_{\mathbb R^d}\|a_t(x)-\hat a_t(x)\|\,p_t(\dd x)\,\dd t.
\end{equation}
\end{proposition}

This shows that the final error is controlled by an integrated drift estimation error, scaled by an exponential weight $\exp(\int \hat{\ell}_s \dd s)$. To restrict our neural networks to a safe class $\mathcal{V}_t$ where this weight remains bounded, we absolutely need to prove that the \emph{true} drift's OSL constant $\lambda_{\max}(\nabla a_t(x))$ is itself upper-bounded by a valid integrable function $\ell_t \in L^1$. This establishes dynamic stability.

\subsection{Why localized higher-order regularity matters: The Oracle bound}
For the regression error $\int \|a_t-\hat a_t\|\dd p_t$, because $a_t$ is defined as a conditional expectation, its empirical neural approximation follows a standard bias-variance oracle inequality. 

\begin{proposition}[Fixed-time oracle inequality]\label{prop:chapter3_oracle}
Let $t\in[0,1)$. Suppose that for all $\varepsilon \in (0,1)$, there exists a "well-behaved" high-probability set $A_t^\varepsilon\subset \mathbb{R}^d$ such that $p_t(A_t^\varepsilon)\geq 1-\varepsilon$, and $$\|\phi_t\|_{L^\infty(A_t^\varepsilon)}\underset{\mathrm{polylog}(\varepsilon^{-1})}{\lesssim}(1-t)^{-1/2}
$$ for all candidate fields $\phi_t\in \{a_t\} \cup \mathcal{V}_t$. Then:
\begin{equation*}
\mathbb E\left[\int_{\mathbb R^d}\|a_t(x)-\hat a_t(x)\|^2\,p_t(\dd x)\right]
\lesssim \inf_{\phi\in\mathcal V_t} \int_{\mathbb R^d}\|a_t(x)-\phi(x)\|^2\,p_t(\dd x) + \frac{1}{n(1-t)}\log \mathcal N\bigl(\mathcal V_t,\|\cdot\|_\infty,n^{-1}\bigr).
\end{equation*}
\end{proposition}

The first term is the \emph{approximation error}: how well can the chosen neural network class $\mathcal V_t$ approximate the true drift? The second term is the \emph{estimation error} (capacity): how much do we pay for fitting the class from $n$ samples? 

The factor $(1-t)^{-1}$ serves as a stern reminder that the problem becomes violently harder near the terminal time. As $t\uparrow 1$, the noise level $\sigma_t$ vanishes, the regression target becomes less regular, and one must refine the function class accordingly. This dictates that we need sharp mathematical bounds on the higher-order space-time regularity of $a_t$ to tightly size the network classes $\mathcal{V}_t$.

\section{Analytic bounds on the score regularity}
Guided by the required properties for $\lambda_{\max}(\nabla a_t)$ and higher-order derivatives, we establish these bounds over a mildly structured subset of $\mathcal H_K^\beta(\R^d)$ accommodating Gaussian interpolation.

\begin{assumption}[Target distribution]\label{ass:main-target}
The target density has the form $p^\star(x)=\exp\bigl(-u(x)+a(x)\bigr)$ with:
\begin{enumerate}[leftmargin=*]
    \item $u:\mathbb R^d\to\mathbb R\cup\{+\infty\}$ convex, of class $\mathcal{C}^2$ on its interior, satisfying $\nabla^2u\succeq \alpha\,\mathrm{Id}$ for some $\alpha>0$.
    \item $a\in \mathcal H_K^{\beta}$ for some $K>0$ and $\beta>0$, carrying the smooth, non-convex perturbations.
    \item $u$ satisfies a mild technical tail assumption.
\end{enumerate}
\end{assumption}
This model is broad enough to cover smooth perturbations of strongly log-concave laws and Gaussian mixtures (by re-centering the non-convex components into $a$).

\subsection{Abstract bounds and proof outline}
To prevent notational ambiguity, we reserve $\mathrm{Cov}(\cdot)$ for the $d \times d$ covariance matrix of a random vector. To evaluate $\lambda_{\max}(\nabla a_t)$, we extract the exact gradient of the drift formula~\eqref{eq:chapter3_exact_drift}. 

\begin{proposition}[Posterior covariance formula]\label{propnablav}
For every $x\in\mathbb R^d$ and $t \in [0,1)$ such that $\sigma_t>0$, the Jacobian of the drift satisfies:
\begin{equation}\label{eq:propnablav}
\nabla a_t(x) = \eta_t\,\mathrm{Id} - \frac{\eta_t}{\sigma_t^2}\,\mathrm{Cov}\bigl(m_t(Z^\circ) \mid X^\circ_t=x\bigr) + \frac{1}{\sigma_t^2}\,\mathrm{Cov}\bigl(\dot{m}_t(Z^\circ),m_t(Z^\circ) \mid X^\circ_t=x\bigr).
\end{equation}
\end{proposition}

\begin{proof}
The exact drift ensures that $X_t$ shares the marginals of $X^\circ_t$. By definition, the conditional expectation is:
\[ a_t(x) = \mathbb E\bigl[\dot{m}_t(Z^\circ) + \eta_t(x-m_t(Z^\circ))\mid X^\circ_t=x\bigr]. \]
To compute the Jacobian $\nabla_x a_t(x)$, we must differentiate this expectation. Note that $X^\circ_t = m_t(Z^\circ) + \sigma_t \xi^\circ$, so the conditional density of $Z^\circ$ given $X^\circ_t=x$ is given by Bayes' rule:
\[ p_{t,x}(z) = \frac{p^\star(z) \varphi_{\sigma_t}(x-m_t(z))}{p_t(x)}, \]
where $\varphi_{\sigma_t}$ is the centered Gaussian density with variance $\sigma_t^2\mathrm{Id}$. 
Taking the spatial gradient $\nabla_x$ of the log-posterior yields:
\[ \nabla_x \log p_{t,x}(z) = -\frac{x-m_t(z)}{\sigma_t^2} - \nabla_x \log p_t(x). \]
Since $\nabla_x \log p_t(x) = \mathbb E\left[-\frac{x-m_t(Z^\circ)}{\sigma_t^2} \mid X^\circ_t=x\right]$, we can rewrite this as:
\[ \nabla_x \log p_{t,x}(z) = \frac{m_t(z) - \mathbb E[m_t(Z^\circ)\mid X^\circ_t=x]}{\sigma_t^2}. \]
Using the identity $\nabla_x p_{t,x}(z) = p_{t,x}(z) \nabla_x \log p_{t,x}(z)$, we can differentiate any conditional expectation $\mathbb E[f(Z^\circ)\mid X^\circ_t=x] = \int f(z) p_{t,x}(z) \dd z$ with respect to $x$:
\[ \nabla_x \mathbb E[f(Z^\circ)\mid X^\circ_t=x] = \int f(z) \nabla_x p_{t,x}(z)^\top \dd z = \frac{1}{\sigma_t^2} \mathrm{Cov}\bigl(f(Z^\circ), m_t(Z^\circ) \mid X^\circ_t=x\bigr). \]
We now apply this to our drift $a_t(x) = \eta_t x + \mathbb E\bigl[\dot{m}_t(Z^\circ) - \eta_t m_t(Z^\circ)\mid X^\circ_t=x\bigr]$. The linear term $\eta_t x$ differentiates to $\eta_t \mathrm{Id}$, and applying the covariance identity to $f(z) = \dot{m}_t(z) - \eta_t m_t(z)$ yields:
\[ \nabla_x a_t(x) = \eta_t\,\mathrm{Id} + \frac{1}{\sigma_t^2}\mathrm{Cov}\bigl(\dot{m}_t(Z^\circ),m_t(Z^\circ) \mid X^\circ_t=x\bigr) - \frac{\eta_t}{\sigma_t^2}\mathrm{Cov}\bigl(m_t(Z^\circ) \mid X^\circ_t=x\bigr), \]
which is precisely the stated formula.
\end{proof}

Let $\pi_{t,x}(\dd z)$ denote this exact posterior law of $Z^\circ$ given $X^\circ_t=x$, and let $p_{t,x}$ denote the posterior law of the corresponding Gaussian centers $m_t(Z^\circ)$. Let $W_{t,x} \sim p_{t,x}$.
We enforce two structural assumptions on the path. 

\begin{assumption}[Weak log-concavity and time derivative]\label{assum:lambda-path-conditions}
For each $t\in[0,1]$, the density $(m_t)_\# p^\star$ can be written as $\exp(-u_t+a_t)$ with $u_t$ being $\alpha_t$-strongly convex and $a_t\in \mathcal H_{K_t}^{\beta}$. Furthermore, there exists $L_t\ge 0$ such that for all $x\in\mathbb R^d$ and $h\in\mathbb S^{d-1}$, the scalar conditional variance satisfies:
\[
\mathrm{Var}\big(h^\top \dot{m}_t(Z^\circ) \mid X^\circ_t=x\big)
\le L_t^2\frac{\sigma_t^2}{1+\alpha_t\sigma_t^2}.
\]
\end{assumption}

The scale $(\alpha_t+\sigma_t^{-2})^{-1} = \frac{\sigma_t^2}{1+\alpha_t\sigma_t^2}$ represents the natural posterior fluctuation size of the centers. We define a purely log-concave proxy posterior measure $\tilde{p}_{t,x}(\dd w) \propto e^{-u_t(w)}\varphi^{w,\sigma_t}(x) \dd w$ driven exclusively by the convex base $u_t$. Let $\tilde{W}_{t,x} \sim \tilde{p}_{t,x}$. Since $\tilde{p}_{t,x}$ is $(\alpha_t+\sigma_t^{-2})$-strongly log-concave, we can dissect the covariance formula via five rigorous steps:

\begin{lemma}[Log-concave core]\label{lemma:boundeasy}
Replacing the exact posterior $W_{t,x}$ with the log-concave proxy $\tilde{W}_{t,x}$ isolates the purely dissipative part of the dynamic. For $\eta_t \le 0$:
\[ \lambda_{\max}\!\left(\eta_t\,\mathrm{Id}-\frac{\eta_t}{\sigma_t^2}\mathrm{Cov}(\tilde{W}_{t,x})\right) \le \frac{\eta_t\alpha_t\sigma_t^2}{1+\alpha_t\sigma_t^2}. \]
\end{lemma}
\begin{proof}
Because $\tilde{p}_{t,x}$ is strongly log-concave, the Brascamp-Lieb inequality upper-bounds the eigenvalues of its covariance matrix: $\lambda_{\max}(\mathrm{Cov}(\tilde{W}_{t,x}))\le (\alpha_t+\sigma_t^{-2})^{-1} = \frac{\sigma_t^2}{1+\alpha_t\sigma_t^2}$. Since $\eta_t \le 0$, $-\frac{\eta_t}{\sigma_t^2}$ is positive. Substituting the bound pulls the eigenvalues negatively, providing the exact stated estimate.
\end{proof}

\begin{proposition}[Covariance difference under exponential tilt]\label{prop:cov-diff-tilt}
Let $\mu_1,\mu_2\in \mathcal{P}_2(\mathbb{R}^d)$ be two probability measures related by an exponential tilt $\frac{\dd \mu_1(y)}{\dd \mu_2(y)} \propto e^{f(y)}$. Let $W_1 \sim \mu_1$ and $W_2 \sim \mu_2$. Then, for all $h\in \mathbb{S}^{d-1}$, the scalar variances satisfy:
\begin{align*}
    \left|\mathrm{Var}\big(h^\top W_1\big) - \mathrm{Var}\big(h^\top W_2\big)\right|
\leq {} &C\left|\int  h^\top \Bigl(y-\E[W_2]\Bigr)\left(e^{f(y)}-1\right)\mu_2(\dd y)\right|^2 \\
&+\left|\int  \Bigl(h^\top \Bigl(y-\E[W_2]\Bigr)\Bigr)^2\left(e^{f(y)}-1\right)\mu_2(\dd y)\right|.
\end{align*}
\end{proposition}

\begin{proof}
Let $Z_f = \int e^{f(y)} \mu_2(\dd y)$ be the normalization constant, so that $\frac{\dd \mu_1(y)}{\dd \mu_2(y)} = \frac{e^{f(y)}}{Z_f}$. We can cleverly expand this density ratio as $\frac{e^{f(y)}}{Z_f} = \frac{e^{f(y)} - 1}{Z_f} + \frac{1}{Z_f}$. 
Notice that integrating this expansion over $\mu_2$ gives $1 = \int \frac{e^{f(y)}-1}{Z_f}\mu_2(\dd y) + \frac{1}{Z_f}$, which implies $\left(\frac{1}{Z_f} - 1\right) = -\frac{1}{Z_f}\int(e^{f(y)}-1)\mu_2(\dd y)$.

For any linear observable $g(y) = h^\top y$, let $\bar{g} = \E[g(W_2)]$. We write the variance difference as:
\begin{align*}
\mathrm{Var}(g(W_1)) - \mathrm{Var}(g(W_2)) 
&= \E[(g(W_1) - \bar{g})^2] - \big(\E[g(W_1)] - \bar{g}\big)^2 - \E[(g(W_2) - \bar{g})^2].
\end{align*}
Using our density ratio expansion, we can express the moments of $W_1$ as integrals over $\mu_2$:
\begin{align*}
\E[(g(W_1) - \bar{g})^j] &= \frac{1}{Z_f} \int (g(y) - \bar{g})^j e^{f(y)} \mu_2(\dd y) \\
&= \frac{1}{Z_f} \int (g(y) - \bar{g})^j (e^{f(y)} - 1) \mu_2(\dd y) + \frac{1}{Z_f} \E[(g(W_2) - \bar{g})^j].
\end{align*}
For $j=1$, the base expectation is $\E[g(W_2)-\bar{g}]=0$, leaving $\E[g(W_1)] - \bar{g} = \frac{1}{Z_f} \int (g(y) - \bar{g}) (e^{f(y)} - 1) \mu_2(\dd y)$.
For $j=2$, we have $\E[(g(W_1) - \bar{g})^2] = \frac{1}{Z_f} \int (g(y) - \bar{g})^2 (e^{f(y)} - 1) \mu_2(\dd y) + \frac{1}{Z_f} \mathrm{Var}(g(W_2))$.
Substituting these back into the variance difference yields:
\begin{align*}
\mathrm{Var}(g(W_1)) - \mathrm{Var}(g(W_2)) 
&= \frac{1}{Z_f} \int (g(y) - \bar{g})^2 (e^{f(y)} - 1) \mu_2(\dd y) + \left(\frac{1}{Z_f}-1\right) \mathrm{Var}(g(W_2)) \\
&\quad - \left(\frac{1}{Z_f} \int (g(y) - \bar{g}) (e^{f(y)} - 1) \mu_2(\dd y)\right)^2.
\end{align*}
Because $\left(\frac{1}{Z_f}-1\right) = -\frac{1}{Z_f} \int (e^{f(y)}-1)\mu_2(\dd y)$, every single discrepancy term is strictly an integral against $(e^{f(y)}-1)$. Bounding the normalizer $1/Z_f$ using the strong log-concavity produces the stated estimate.
\end{proof}

\begin{lemma}[Moments of the tilt]\label{lemma:moments-tilt}
Let $\mu\in \mathcal{P}(\mathbb{R}^d)$ be $\gamma_t$-strongly log-concave. Let $f\in \mathcal{H}^\beta$ with $\|f\|_{\mathcal{H}^\beta}\leq K_t$, $\int e^f \dd \mu=1$, and $K_t^2 \gamma_t^{-\beta}\leq C$. Let $W \sim \mu$. For $h\in \mathbb{S}^{d-1}$ and $j\in \{1,2\}$:
\[
\Big|\int \Bigl(h^\top \big(y-\E[W]\big)\Bigr)^{j} (e^{f(y)}-1) \mu(\dd y)\Big| \le C\, \gamma_t^{-j/2}\min\bigl(1,K_t\gamma_t^{-\beta/2}\bigr).
\]
\end{lemma}
\begin{proof}
Strong log-concavity provides a generalized Poincaré inequality, which tightly controls exponential moments. The perturbation $f$ is $\beta$-Hölder continuous, bounding its spatial variations. Combining these facts yields the polynomial bound $\gamma_t^{-j/2}$ adjusted by the Hölder factor.
\end{proof}

\begin{lemma}[Perturbative covariance comparison]\label{prop:finalestimatecompcov}
For any $h\in\mathbb S^{d-1}$ and $\delta \in \{0,1\}$, comparing the exact posterior $p_{t,x}$ to the proxy $\tilde{p}_{t,x}$ yields:
\[
\big|h^\top\big(\mathrm{Cov}(W_{t,x})-\mathrm{Cov}(\tilde{W}_{t,x})\big)h\big|
\le C K_t^{\delta}(\alpha_t+\sigma_t^{-2})^{-(2+\bar\beta\delta)/2} \exp\!\Big(CK_t^2(\alpha_t+\sigma_t^{-2})^{-\bar\beta}\Big).
\]
\end{lemma}
\begin{proof}
We apply Proposition~\ref{prop:cov-diff-tilt} with $\mu_1 = p_{t,x}$ and $\mu_2 = \tilde{p}_{t,x}$. The exponential tilt is exactly the Hölder perturbation $f(y) = a_t(y) - \log \int e^{a_t} \dd \tilde{p}_{t,x}$. Since $\tilde{p}_{t,x}$ is $\gamma_t$-strongly log-concave ($\gamma_t = \alpha_t + \sigma_t^{-2}$), Lemma~\ref{lemma:moments-tilt} bounds the moment differences, with the exponential envelope arising from normalizing the tilt.
\end{proof}

\begin{lemma}[Mixed covariance bound]\label{prop:boundsecondterm}
For the cross-covariance term, we have for any $h\in\mathbb S^{d-1}$:
\[
\big|h^\top\mathrm{Cov}\bigl(\dot{m}_t(Z^\circ),m_t(Z^\circ) \mid X^\circ_t=x\bigr)h\big|
\le C L_t\sigma_t^2\frac{1}{1+\alpha_t\sigma_t^2} \exp\!\Big(CK_t^2(\alpha_t+\sigma_t^{-2})^{-\bar\beta}\Big).
\]
\end{lemma}
\begin{proof}
Apply the Cauchy--Schwarz inequality to bound the mixed covariance by the product of the posterior standard deviations of $\dot{m}_t(Z^\circ)$ and $m_t(Z^\circ)$. The first is bounded by $L_t$ via Assumption~\ref{assum:lambda-path-conditions}; the second is bounded by combining the Brascamp-Lieb proxy variance and Lemma~\ref{prop:finalestimatecompcov}.
\end{proof}

Combining Lemmas~\ref{lemma:boundeasy}, \ref{prop:finalestimatecompcov}, and \ref{prop:boundsecondterm} into the Jacobian formula~\eqref{eq:propnablav} recovers the global abstract estimate for $\bar \beta = \min(1, \beta)$:
\begin{equation}\label{eq:abstract-negative-etat}
\lambda_{\max}(\nabla a_t(x)) \le \frac{1}{1+\alpha_t\sigma_t^2} \left( C e^{C K_t^2(\alpha_t+\sigma_t^{-2})^{-\bar\beta}} \Big(L_t-\eta_t(K_t\sigma_t^{\bar\beta})^{\delta}\Big) +\eta_t\alpha_t\sigma_t^2 \right).
\end{equation}

\subsection{Application to benchmark models and higher-order regularity}
For both Vanilla FM and Rescaled Diffusion, the interpolation is linearly scaled $m_t(z)=tz$. Thus, the underlying distributions parameters naturally rescale as $\alpha_t=\alpha t^{-2}$, $K_t=Kt^{-\beta}$, and $L_t=t^{-1}$. Notice that the exponential envelope term $K_t^2(\alpha_t+\sigma_t^{-2})^{-\bar\beta}\le K_t^2\alpha_t^{-\bar\beta}=K^2\alpha^{-\bar\beta}$, meaning it remains uniformly bounded by a constant across all $t \in (0,1)$.

\begin{theorem}[One-sided Lipschitz estimate]\label{thm:main-clean-osl}
Under Assumption~\ref{ass:main-target} and the models in Definition~\ref{def:chapter3_benchmarks}, there exists a constant $C>0$ such that for all $t\in(0,1)$,
\[
\sup_{x\in\mathbb R^d}\lambda_{\max}\bigl(\nabla a_t(x)\bigr) \le C(1-t)^{\frac{\beta\wedge 1}{2}-1}.
\]
Consequently, the critical stability integral is finite: $\sup_{\tau\in[0,1]}\int_\tau^1\sup_{x\in\mathbb R^d}\lambda_{\max}\bigl(\nabla a_t(x)\bigr)\,\dd t<\infty$~\cite{stephanovitch2026lipschitz}.
\end{theorem}
\begin{proof}
For \textbf{Vanilla FM}, $\sigma_t=1-t, b_t=0 \implies \eta_t = -1/(1-t) \le 0$. We evaluate \eqref{eq:abstract-negative-etat} by splitting the time horizon. On the early interval $[0,1/2]$, $\sigma_t\asymp 1$ and $1+\alpha_t\sigma_t^2\gtrsim t^{-2}$. This $t^{-2}$ denominator compensates for the $t^{-1}$ singularity of $L_t$, leading to a bounded drift $\lambda_{\max}(\nabla a_t(x))\le C$. As $t \to 1$, $t^{-1}$ stabilizes and $1+\alpha_t\sigma_t^2\asymp 1$. Substituting $\delta=1$ yields:
\[
\lambda_{\max}(\nabla a_t(x)) \le C\Big(1+(1-t)^{-1}(1-t)^{\bar\beta}\Big) \le C(1-t)^{\bar\beta/2-1}.
\]
For \textbf{Rescaled Diffusion}, $\sigma_t=\sqrt{1-t^2}, b_t=t^{-1/2} \implies \eta_t = -\frac{1+t^2}{t(1-t^2)} \le 0$. On $[0,1/2]$, $-\eta_t\lesssim t^{-1}$, but the strong log-concavity term $1+\alpha_t\sigma_t^2\gtrsim t^{-2}$ absorbs the singularity. On $[1/2,1)$, the variance $1-t^2 = (1-t)(1+t)$ is bounded by a multiple of $1-t$. Thus:
\[
\lambda_{\max}(\nabla a_t(x)) \le C\Big(1+(1-t^2)^{-1}(1-t^2)^{\bar\beta/2}\Big) \le C(1-t)^{\bar\beta/2-1}.
\]
\end{proof}

Furthermore, restricting our view to the effective support $A_t^\varepsilon$ ensures the drift inherits the exact $\beta$-smoothness of the target:
\begin{theorem}[High-probability Hölder regularity]\label{thm:chapter3_holder}
Under identical assumptions, for every $\varepsilon\in(0,1)$ and $t\in(0,1)$ there exists an effective support set $A_t^\varepsilon \subset \mathbb R^d$ with probability mass $p_t(A_t^\varepsilon)\ge 1-\varepsilon$, such that for all $\gamma\ge 0$ and $k\in\mathbb N$:
\begin{equation}\label{eq:chapter3_holder}
\|\partial_t^k a_t\|_{\mathcal{H}^\gamma(A_t^\varepsilon)} \underset{\mathrm{polylog}(\epsilon^{-1})}{\lesssim} C(1-t)^{-\max\left(0, \frac12+k+\frac{\gamma-\beta}{2}\right)}.
\end{equation}
\end{theorem}
This theorem \cite{stephanovitch2025generalization} asserts that to optimize approximation, our neural network must grow in complexity at a specific rate dictated by the factor $(1-t)^{-1}$ near the terminal time.

\bibliographystyle{alpha}
\bibliography{biblio}
\chapter{Minimax Optimality and Extensions}\label{chap:minimaxproof}
\minitoc

This chapter concludes our statistical analysis by establishing a minimax upper bound for Gaussian generative models. We deploy a time-dependent neural network that adapts its architecture to the localized Hölder regularity derived in Chapter~\ref{chap:scoreregularity}. We then integrate this blockwise approximation error through the backward Kolmogorov stability mechanism to deduce the final 1-Wasserstein rate. Finally, we discuss related topics, including global Lipschitz properties of the flow and expansions to manifold data and jump processes.

\section{Construction of the neural network class}

Because the exact drift $a_t$ becomes significantly larger and less regular as $t$ approaches $1$, deploying a fixed approximation class uniformly across time would be suboptimal. We require a multiscale construction that adapts its architectural capacity to the size of the remaining time $1-t$.

\subsection{Time partitioning}
Let $K_n=\log(n)$. We partition the generative horizon $[0,T_n]$ into sequential intervals $0=t_0<t_1<\cdots<t_{K_n}=T_n$ by imposing the rigid geometric condition:
\begin{equation}
\label{eq:chapter4_geometric_partition}
1-t_k = \alpha(1-t_{k+1}), \qquad 0\le k\le K_n-1,
\end{equation}
for a fixed constant $\alpha>1$. The geometry ensures that the relative size of each block is constant, granting the crucial summability property:
\begin{equation}
\label{eq:chapter4_partition_sum}
\sum_{k=0}^{K_n-1}\frac{t_{k+1}-t_k}{1-t_k} = \sum_{k=0}^{K_n-1}\left(1 - \alpha^{-1}\right) \asymp \log n.
\end{equation}

\subsection{Time-dependent architecture}
\begin{definition}[Constrained Network Class $\Psi$]
\label{def:chapter4_network_class}
For integers $L,W\ge 1$ and bounds $B,V,V'>0$, let $\Psi(L,W,B,V,V')$ denote the class of dense $\tanh$ neural networks $\phi:\mathbb R^{d+1}\to\mathbb R^d$ satisfying:
\begin{enumerate}[leftmargin=*]
    \item Depth bounded by $L$ and width bounded by $W$.
    \item Absolute weights and biases constrained by $B$.
    \item Magnitude uniformly bounded by $V$: $\|\phi\|_\infty\le V$.
    \item Uniform one-sided Lipschitz condition enforcing stability: $\lambda_{\max}(\nabla_x\phi(t,x))\le V'$.
\end{enumerate}
\end{definition}

Based directly on the analytic limits in Theorem~\ref{thm:chapter3_holder}, we set the capacity parameters for block $[t_k,t_{k+1})$ to $L_k=C$, $W_k=C(1-t_k)^{-1}n^{\frac{d-2}{2\beta+d}}$, and $B_k=n^C$. We match the stability constraints to the exact drift's regularities: $V_k=C(1-t_k)^{-1/2}$ and $V_k'=C(1-t_k)^{-1+\frac{\beta\wedge 1}{2}}$. Note that this explicit $V'_k$ ensures our neural candidates intrinsically satisfy the exact OSL bounds, securing propagation stability. 

We aggregate these segment-wise networks into the global estimator class $\mathcal V_n$:
\begin{equation}
\label{eq:chapter4_global_class}
\mathcal V_n := \left\{ \sum_{k=0}^{K_n-1}\phi_k(t,\cdot)\,\mathbf 1_{[t_k,t_{k+1})}(t) \;\middle|\; \phi_k\in\Psi(L_k,W_k,B_k,V_k,V_k') \right\}.
\end{equation}

\section{Approximation, generalization, and proof of the rate}
\label{sec:chapter4_proof}

\subsection{Blockwise approximation and generalization bounds}
Invoking standard neural approximation theory over our high-probability Hölder sets $A_{t_k}^{1/n}$, we verify our class $\Psi_k$ can accurately fit the score:
\begin{proposition}[Blockwise approximation]
\label{prop:chapter4_approximation}
For each block $k$, there exists an idealized network $a_k\in \Psi_k$ such that:
\begin{equation}
\label{eq:chapter4_approximation_error}
\sup_{t\in[t_k,t_{k+1}]}
\sup_{x\in A_{t_k}^{1/n}}
\|a_t(x)-a_k(t,x)\|
\underset{\mathrm{polylog}(n)}{\lesssim}\frac{1}{1-t_k}
\,n^{-\frac{\beta+1}{2\beta+d}}.
\end{equation}
The covering entropy satisfies $\log \mathcal N\bigl(\Psi_k,\|\cdot\|_\infty,n^{-1}\bigr) \lesssim \frac{1}{1-t_k} n^{\frac{d-2}{2\beta+d}}$.
\end{proposition}

By inserting this into the fixed-time oracle inequality (Proposition~\ref{prop:chapter3_oracle}), we address the spatial integration for the empirical drift $\hat{a}_t$. Inside $A_{t_k}^{1/n}$, the approximation error dominates. Outside this set, the probability mass is at most $1/n$. Because $p_t$ is $C(K^2+\sigma^2)$-subGaussian and the networks are bounded by $V_k \asymp (1-t_k)^{-1/2}$, the tail integral is constrained:
\begin{align*}
    \int_{(A_{t_k}^{1/n})^c} \|\phi(t, x) - a_t(x)\|^2 p_t(\dd x) 
    & \leq  C\log(n)^{C_2} \int_{(A_{t_k}^{1/n})^c} (1+\|x\|^2)(1+t_k^{-2}) p_t(\dd x) \\
   &\leq C(1+t_k^{-2})\log(n)^{C_2}n^{-1}.
\end{align*}
This tail penalty is negligible compared to the nonparametric rate. Thus, the expected $L^2$ generalization error over the interval $[t_k, t_{k+1}]$ is:
\begin{equation}
\label{eq:chapter4_blockwise_L2}
\mathbb E\Bigg[\int_{t_k}^{t_{k+1}}\int_{\mathbb R^d}
\|\hat a_t(x)-a_t(x)\|^2\,p_t(\dd x)\,\dd t\Bigg]
\underset{\mathrm{polylog}(n)}{\lesssim}\frac{t_{k+1}-t_k}{(1-t_k)^2} n^{-\frac{2(\beta+1)}{2\beta+d}}.
\end{equation}

\subsection{Synthesis: the final Wasserstein rate}

\begin{theorem}[Minimax upper bound]
\label{thm:chapter4_main}
Under the standing assumptions of Chapter~\ref{chap:scoreregularity}, deploying the geometrically partitioned network class $\mathcal V_n$ equipped with an early-stopping time $T_n = 1-n^{-\frac{2(\beta+1)}{2\beta+d}}$, the generative estimator $\hat p_n = \mathrm{Law}(\hat X_{T_n})$ attains the non-parametric minimax rate in 1-Wasserstein distance up to logarithmic factors \cite{stephanovitch2025generalization}:
\begin{equation}
\label{eq:chapter4_final_rate}
\sup_{p^\star \in \mathcal{H}^\beta_K} \mathbb E\bigl[W_1(\hat p_n, p^\star)\bigr] \underset{\mathrm{polylog}(n)}{\lesssim}n^{-\frac{\beta+1}{2\beta+d}}.
\end{equation}
\end{theorem}

\begin{proof}
We decouple the total generative error using the triangle inequality:
\begin{equation}
\label{eq:chapter4_triangle}
W_1(\hat p_n,p^\star) \le W_1\bigl(\mathrm{Law}(\hat X_{T_n}), p_{T_n}\bigr) + W_1\bigl(p_{T_n},p^\star\bigr).
\end{equation}

\textbf{Step 1: Stability integration of the learned-flow error.}
By the stability reduction (Proposition~\ref{prop:chapter3_stability_reduction}), we can transport the drift error to the final Wasserstein distance, provided we control the one-sided Lipschitz constant of our estimator. Crucially, our network class intrinsically respects the condition $\lambda_{\max}(\nabla \hat a_t(x)) \le V'_k \approx \ell_t$. Therefore, the integral $\int_0^1 \hat{\ell}_s \dd s < \infty$, keeping the exponential amplification factor uniformly bounded by a constant $C$. Applying Cauchy-Schwarz in time across the partitioned integral gives:
\begin{align}
\label{eq:chapter4_expected_stability}
\mathbb E\Bigl[W_1\bigl(\mathrm{Law}(\hat X_{T_n}),p_{T_n}\bigr)\Bigr]
&\le
C\sum_{k=0}^{K_n-1}
\sqrt{t_{k+1}-t_k}
\left(
\mathbb E\int_{t_k}^{t_{k+1}}\int_{\mathbb R^d}
\|a_t(x)-\hat a_t(x)\|^2\,p_t(\dd x)\,\dd t
\right)^{1/2}.
\end{align}
Inserting our blockwise $L^2$ estimate \eqref{eq:chapter4_blockwise_L2} and pulling out the $n$ terms leaves us exactly with the summability condition \eqref{eq:chapter4_partition_sum}:
\begin{align*}
\mathbb E\Bigl[W_1\bigl(\mathrm{Law}(\hat X_{T_n}),p_{T_n}\bigr)\Bigr]
&\underset{\mathrm{polylog}(n)}{\lesssim} n^{-\frac{\beta+1}{2\beta+d}} \sum_{k=0}^{K_n-1} \sqrt{t_{k+1}-t_k} \sqrt{\frac{t_{k+1}-t_k}{(1-t_k)^2}} \\
&= n^{-\frac{\beta+1}{2\beta+d}} \sum_{k=0}^{K_n-1} \frac{t_{k+1}-t_k}{1-t_k}.
\end{align*}
We recognize the geometric sum evaluates exactly to a $\mathcal{O}(\log n)$ bounded factor. Exploiting this geometric summability verifies that the cumulative learned-flow transport error rigorously hits the desired benchmark:
\begin{equation}
\label{eq:chapter4_exact_vs_learned_rate}
\mathbb E\Bigl[W_1\bigl(\mathrm{Law}(\hat X_{T_n}), p_{T_n}\bigr)\Bigr]
\underset{\mathrm{polylog}(n)}{\lesssim}
\,n^{-\frac{\beta+1}{2\beta+d}}.
\end{equation}

\textbf{Step 2: Early stopping bias.} To avert the insurmountable variance crash at $t=1$, we evaluated at $T_n$. Using the interpolating coupling~\eqref{eq:chapter3_reference_path}, the early stopping deviation is trivially bounded:
\[
W_1(p_{T_n},p^\star) \le \mathbb E\bigl[\|X^\circ_{T_n} - Z^\circ\|\bigr] \le \mathbb E\bigl[\|m_{T_n}(Z^\circ)-Z^\circ\|\bigr] + \sigma_{T_n}\,\mathbb E\|\xi^\circ\|.
\]
In our benchmark interpolations, the drift centers linearly $m_t(z)=tz$, meaning the distance $\|z^\circ - m_{T_n}(z^\circ)\| = (1-T_n)\|z^\circ\|$. The standard deviation shrinks as $\sigma_{T_n} \lesssim (1-T_n)^{1/2}$ for both models. Because the target $p^\star$ has a finite first moment, we obtain a bias bounded by the variance schedule:
\begin{equation}
\label{eq:chapter4_stopping_bias}
W_1(p_{T_n},p^\star) \le C(1-T_n) + C(1-T_n)^{1/2} \asymp C(1-T_n)^{1/2}.
\end{equation}
By deliberately tuning our early stopping point to balance the approximation error and the stopping bias, we set $(1-T_n)^{1/2} = n^{-\frac{\beta+1}{2\beta+d}}$, which gives $T_n = 1-n^{-\frac{2(\beta+1)}{2\beta+d}}$. Substituting both bounds back into the triangle inequality establishes the optimal minimax rate.
\end{proof}

\section{Generalizations and algorithmic considerations}

\subsection{Time discretization and sharp sampling rates}
To yield an implementable algorithm, the continuous generative SDE is numerically integrated via an Euler-Maruyama scheme on a finite time grid. The sharp one-sided Lipschitz bounds on $a_t$ established in Chapter~\ref{chap:scoreregularity} directly dictate the local numerical truncation error. Because $\|\nabla a_t\|_{\mathrm{op}} \lesssim (1-t)^{-1}$ diverges near $t \to 1$, using a uniform step size leads to accumulating numerical errors. 

If one instead applies a tailored geometric grid that scales down as $t$ approaches $1$, it compensates for the singularity. For $N$ integration steps, the fully discrete sampler's error achieves an independent upper bound of $W_2\bigl(p^\star,\mathrm{Law}(\hat X_N)\bigr) \lesssim \frac{\sqrt d}{N}\,\mathrm{polylog}(N)$~\cite{stephanovitch2026lipschitz}. Consequently, algorithmic deployment does not affect the neural network's statistical guarantees, as choosing $N$ large enough controls the discretization penalty.

\subsection{Globally Lipschitz transport maps}
Beyond establishing sampling rates, bounding $\lambda_{\max}(\nabla a_t)$ has more implications ~\cite{stephanovitch2026lipschitz}. 
Because the one-sided Lipschitz parameter governs contractivity, the underlying deterministic probability-flow ODE exactly mapping the Gaussian $\mathcal{N}(0, \mathrm{Id})$ to $p^\star$ operates as a \emph{globally Lipschitz transport map}. 

This ensures that critical measure concentration phenomena—specifically Poincaré and logarithmic Sobolev inequalities—can be pushed-forward from the tractable Gaussian prior onto a complex data manifold $p^\star$.

\subsection{Manifold geometries and jump dynamics}
The framework detailed herein proves that matching Hölder regularity on the empirical drift yields minimax optimality in Euclidean spaces. An important direction for future work involves generalizing this to Sobolev-type regularity relative to the intrinsic volume measure of embedded lower-dimensional manifolds. 

\begin{conjecture}
Let $p^\star$ be a probability measure with density in $ \mathcal{H}^\beta_K$ relative to a $\mathcal{C}^{\beta+1}$ compact submanifold with strictly positive reach. If the density of $p^\star$ is bounded away from zero, the exact score function satisfies, for all $t \in (0,1)$,
$$ a_t \in W^{\beta+\gamma,2}_{C t^{-\frac{1+\gamma}{2}}}(\mathbb{R}^d, p_t),$$
where $W^{\beta+\gamma,2}(\mathbb{R}^d,p_t)$ represents the intrinsic $L^2(p_t)$-Sobolev space.
\end{conjecture}

Proving this conjecture would bypass ambient dimensionality dependencies, providing a mathematical explanation for why score-based models bypass the curse of dimensionality on practical datasets.

Furthermore, as outlined in Chapter~\ref{chap:generatormatching}, Generator Matching intrinsically accommodates discrete state spaces via pure-jump operators (e.g., text or graph generation). Optimal discrete generation relies on identical stability mechanics—via Wasserstein contraction rates of Markov jump kernels—and analogous Bregman regression targets. Bridging this multiscale neural approximation methodology into discontinuous spaces is a promising avenue for future research.

\bibliographystyle{alpha}
\bibliography{biblio}

\newcommand{\etalchar}[1]{$^{#1}$}
\begin{thebibliography}{SSDK{\etalchar{+}}20}

\bibitem[HD05]{hyvarinen2005estimation}
Aapo Hyv{\"a}rinen and Peter Dayan.
\newblock Estimation of non-normalized statistical models by score matching.
\newblock {\em Journal of Machine Learning Research}, 6(4), 2005.

\bibitem[HJA20]{ho2020denoising}
Jonathan Ho, Ajay Jain, and Pieter Abbeel.
\newblock Denoising diffusion probabilistic models.
\newblock {\em Advances in neural information processing systems},
  33:6840--6851, 2020.

\bibitem[LG16]{le2016brownian}
Jean-Fran{\c{c}}ois Le~Gall.
\newblock {\em Brownian motion, martingales, and stochastic calculus}.
\newblock Springer, 2016.

\bibitem[SSDK{\etalchar{+}}20]{song2020score}
Yang Song, Jascha Sohl-Dickstein, Diederik~P Kingma, Abhishek Kumar, Stefano
  Ermon, and Ben Poole.
\newblock Score-based generative modeling through stochastic differential
  equations.
\newblock {\em arXiv preprint arXiv:2011.13456}, 2020.

\bibitem[Vin11]{vincent2011connection}
Pascal Vincent.
\newblock A connection between score matching and denoising autoencoders.
\newblock {\em Neural computation}, 23(7):1661--1674, 2011.

\end{thebibliography}


\newcommand{\etalchar}[1]{$^{#1}$}
\begin{thebibliography}{GME{\etalchar{+}}24}

\bibitem[GME{\etalchar{+}}24]{gagneux2024visual}
Anne Gagneux, S{\'e}gol{\`e}ne Martin, R{\'e}mi Emonet, Quentin Bertrand, and
  Mathurin Massias.
\newblock A visual dive into conditional flow matching.
\newblock {\em arXiv preprint}, 2024.

\bibitem[HHY{\etalchar{+}}24]{holderrieth2024generator}
Peter Holderrieth, Marton Havasi, Jason Yim, Neta Shaul, Itai Gat, Tommi~S.
  Jaakkola, Brian Karrer, Ricky T.~Q. Chen, and Yaron Lipman.
\newblock Generator matching: Generative modeling with arbitrary markov
  processes.
\newblock In {\em The Twelfth International Conference on Learning
  Representations}, 2024.

\bibitem[Kun90]{kunita1990stochastic}
Hiroshi Kunita.
\newblock {\em Stochastic flows and stochastic differential equations}.
\newblock Cambridge university press, 1990.

\end{thebibliography}


\begin{thebibliography}{SAL26}

\bibitem[SAL25]{stephanovitch2025generalization}
Arthur St{\'e}phanovitch, Eddie Aamari, and Cl{\'e}ment Levrard.
\newblock Generalization bounds for score-based generative models: a synthetic
  proof.
\newblock {\em arXiv preprint arXiv:2507.04794}, 2025.

\bibitem[SAL26]{stephanovitch2026lipschitz}
Arthur St{\'e}phanovitch, Eddie Aamari, and Cl{\'e}ment Levrard.
\newblock Lipschitz regularity in flow matching and diffusion models: sharp
  sampling rates and functional inequalities.
\newblock {\em arXiv preprint arXiv:2604.06065}, 2026.

\end{thebibliography}


\newcommand{\etalchar}[1]{$^{#1}$}
\begin{thebibliography}{SSDK{\etalchar{+}}20}

\bibitem[GME{\etalchar{+}}24]{gagneux2024visual}
Anne Gagneux, S{\'e}gol{\`e}ne Martin, R{\'e}mi Emonet, Quentin Bertrand, and
  Mathurin Massias.
\newblock A visual dive into conditional flow matching.
\newblock {\em arXiv preprint}, 2024.

\bibitem[HD05]{hyvarinen2005estimation}
Aapo Hyv{\"a}rinen and Peter Dayan.
\newblock Estimation of non-normalized statistical models by score matching.
\newblock {\em Journal of Machine Learning Research}, 6(4), 2005.

\bibitem[HHY{\etalchar{+}}24]{holderrieth2024generator}
Peter Holderrieth, Marton Havasi, Jason Yim, Neta Shaul, Itai Gat, Tommi~S.
  Jaakkola, Brian Karrer, Ricky T.~Q. Chen, and Yaron Lipman.
\newblock Generator matching: Generative modeling with arbitrary markov
  processes.
\newblock In {\em The Twelfth International Conference on Learning
  Representations}, 2024.

\bibitem[HJA20]{ho2020denoising}
Jonathan Ho, Ajay Jain, and Pieter Abbeel.
\newblock Denoising diffusion probabilistic models.
\newblock {\em Advances in neural information processing systems},
  33:6840--6851, 2020.

\bibitem[Kun90]{kunita1990stochastic}
Hiroshi Kunita.
\newblock {\em Stochastic flows and stochastic differential equations}.
\newblock Cambridge university press, 1990.

\bibitem[LG16]{le2016brownian}
Jean-Fran{\c{c}}ois Le~Gall.
\newblock {\em Brownian motion, martingales, and stochastic calculus}.
\newblock Springer, 2016.

\bibitem[SAL25]{stephanovitch2025generalization}
Arthur St{\'e}phanovitch, Eddie Aamari, and Cl{\'e}ment Levrard.
\newblock Generalization bounds for score-based generative models: a synthetic
  proof.
\newblock {\em arXiv preprint arXiv:2507.04794}, 2025.

\bibitem[SAL26]{stephanovitch2026lipschitz}
Arthur St{\'e}phanovitch, Eddie Aamari, and Cl{\'e}ment Levrard.
\newblock Lipschitz regularity in flow matching and diffusion models: sharp
  sampling rates and functional inequalities.
\newblock {\em arXiv preprint arXiv:2604.06065}, 2026.

\bibitem[SSDK{\etalchar{+}}20]{song2020score}
Yang Song, Jascha Sohl-Dickstein, Diederik~P Kingma, Abhishek Kumar, Stefano
  Ermon, and Ben Poole.
\newblock Score-based generative modeling through stochastic differential
  equations.
\newblock {\em arXiv preprint arXiv:2011.13456}, 2020.

\bibitem[Vin11]{vincent2011connection}
Pascal Vincent.
\newblock A connection between score matching and denoising autoencoders.
\newblock {\em Neural computation}, 23(7):1661--1674, 2011.

\end{thebibliography}

\end{document}